\documentclass[10pt]{article}

\usepackage[affil-it]{authblk}
\usepackage{graphicx}
\usepackage{amsthm,amsmath,amsfonts,amssymb}
\usepackage{stmaryrd}

\newtheorem{thm}{Theorem}[section]
\newtheorem{prop}[thm]{Proposition}

\theoremstyle{definition} 
\newtheorem{defn}{Definition}[section]
\newtheorem{exmp}{Example}[section]

\newcommand{\CL}[1]{\mathcal{#1}}
\newcommand{\RM}[1]{\mathrm{#1}}
\newcommand{\BF}[1]{\mathbf{#1}}
\newcommand{\BB}[1]{\mathbb{#1}}

\newcommand{\ten}[1]{\underline{\mathbf{#1}}}    
\newcommand{\unfoldto}[1]{<#1>}  
\newcommand{\modcon}{\times^1}  
\newcommand{\ptrace}{Tr}  
\newcommand{\skron}{\,|\!\!\otimes\!\!|\,}   
\newcommand{\Trps}{\top}  
\newcommand\myatop[2]{\genfrac{}{}{0pt}{}{#1}{#2}}

\begin{document}


\title{Fundamental Tensor Operations for Large-Scale Data Analysis in Tensor Train Formats}

\author[a]{Namgil Lee}
\affil[a]{Laboratory for Advanced Brain Signal Processing, 
	RIKEN Brain Science Institute, Wako-shi, Saitama, 3510198, Japan}
\author[a]{Andrzej Cichocki}

\date{}

\maketitle


\begin{abstract} 

We discuss extended definitions of linear and multilinear operations 
such as Kronecker, Hadamard, and contracted products, and establish links between 
them for tensor calculus. Then we introduce effective low-rank tensor 
approximation techniques including Candecomp/Parafac (CP), 
Tucker, and tensor train (TT) decompositions
with a number of mathematical and graphical representations. 
We also provide a brief review of mathematical properties of  the 
TT decomposition as a low-rank approximation technique. 
With the aim of breaking the curse-of-dimensionality 
in large-scale numerical analysis, 
we describe basic operations on large-scale vectors, matrices, and 
high-order tensors represented by TT decomposition. 
The proposed representations can be used for describing numerical methods 
based on TT decomposition for solving large-scale optimization problems 
such as systems of linear equations and symmetric eigenvalue problems. 

\vspace{1pc}

\end{abstract} 



\section{Introduction}

Multi-dimensional or multi-way data is prevalent nowadays, which can be represented by tensors. 
An $N$th-order tensor is a multi-way array of size 
$I_1\times I_2\times \cdots \times I_N$, 
where the $n$th dimension or mode is of size $I_n$.
For example, a tensor can be induced by the discretization of a multivariate 
function \cite{Kho2012}. Given a multivariate function 
$f(x_1,\ldots,x_N)$ defined on 
a domain $[0, 1]^N$, we can get a tensor with entries containing the  
function values at grid points. 
For another example, we can obtain tensors based on observed data 
\cite{Cic2009,KolBa2009}. 
We can collect and integrate measurements from different modalities
by neuroimaging technologies such as functional magnetic resonance imaging (fMRI)
and electroencephalography (EEG): 
subjects, time, frequency, electrodes, task conditions, trials, and so on. 
Furthermore, high-order tensors can be created by a process called 
tensorization or quantization \cite{Kho2011}, 
by which a large-scale vectors and matrices are reshaped into higher-order 
tensors.

However, it is impossible to store a high-order tensor because 
the number of entries, $I^N$ when $I=I_1=I_2=\cdots=I_N$, 
grows exponentially as the order $N$ increases. This is called 
the ``curse-of-dimensionality''. Even for $I=2$, with $N=50$ we 
obtain $2^{50}\approx 10^{15}$ entries. 
Such a huge storage and computational costs required 
for high-dimensional problems prohibit the use of standard numerical 
algorithms. To make high-dimensional problems tractable, there have been
developed approximation 
methods including sparse grids \cite{Bun2004,Smo1963} and 
low-rank tensor approximations \cite{Gra2013,Hac2012,Kho2012}. 
In this paper, we focus on the latter approach, where computational 
operations are performed on tensor formats, i.e.,
low-parametric representations of tensors. 

In this paper, we consider several tensor decompositions (formats), especially 
the tensor train (TT) decomposition, which 
is one of the simplest tensor networks developed 
with the aim of overcoming the curse-of-dimensionality. 
Extensive overviews of the modern low-rank tensor approximation 
techniques are presented in \cite{Gra2013,Kho2012}. 
The TT decomposition is equivalent to the matrix product states (MPS) 
for open boundary conditions proposed in computational physics, 
and it has taken a key role in density matrix renormalization group (DMRG) 
methods for simulating quantum many-body systems \cite{Sch2011,Whi1993}. 
It was later re-discovered in numerical analysis community
\cite{Holtz2012,Ose2011,Ose2009}. 
The TT-based numerical algorithms can accomplish 
algorithmic stability and adaptive determination of ranks 
by employing the singular value decomposition (SVD)
\cite{Ose2011}. 
Its scope of application is quickly expanding 
for addressing high-dimensional problems such as 
multi-dimensional integrals, stochastic and parametric PDEs,  
computational finance, and machine learning \cite{Gra2013}. 
On the other hand, a comprehensive survey on traditional 
low-rank tensor approximation techniques such as Candecomp/Parafac (CP) and Tucker 
decompositions is presented in \cite{delath2009,KolBa2009}. 

Despite the large interest in high-order tensors 
in TT format, mathematical representations of the TT 
tensors are usually limited to the representations 
based on scalar operations on matrices and vectors, 
which leads to complex and tedious index notation 
in the tensor calculus. 
For example, a TT tensor is defined by each entry represented as
products of matrices \cite{Holtz2012,Ose2011}. 
On the other hand, representations of traditional 
low-rank tensor decompositions have been developed based on 
multilinear operations such as the Kronecker product, 
Khatri-Rao product, Hadamard product, and mode-$n$ multilinear product 
\cite{Cic2009,KolBa2009}, which enables
coordinate-free notation. 
Through the utilization of the multilinear operations,  
the traditional tensor decompositions expanded the area of 
application to chemometrics, signal processing, 
numerical linear algebra, computer vision, 
data mining, graph analysis, and 
neuroscience \cite{KolBa2009}.


In this work, we develop extended definitions of multilinear 
operations on tensors. 
Based on the tensor operations, we provide 
a number of new and useful representations of the TT decomposition. 
We also provide graphical representations of TT decompositions, 
motivated by \cite{Holtz2012}, 
which are helpful in understanding the underlying 
principles and TT-based numerical algorithms. 
Based on the TT representations of 
large-scale vectors and matrices, 
we show that the basic numerical operations 
such as the addition, contraction, matrix-by-vector product, and 
quadratic form are conveniently described by the proposed 
representations.  
We demonstrate the usefulness of the 
proposed tensor operations in tensor calculus 
by giving a proof of orthonormality of the so-called frame matrices. 
Moreover, we derive explicit representations of localized 
linear maps in TT format that have been implicitly presented in matrix forms 
in the literature in the context of 
alternating linear scheme (ALS) for solving 
various optimization problems. 
The suggested mathematical 
operations and TT representations can be 
further applied to describing TT-based numerical methods 
such as the solutions to large-scale systems of linear equations 
and eigenvalue problems \cite{Holtz2012}. 

This paper is organized as follows. In Section 2, 
we introduce notations for tensors and 
definitions for tensor operations. 
In Section 3, we provide the mathematical and 
graphical representations of the 
TT decomposition. We also review mathematical properties of
the TT decomposition as a low-rank approximation. 
In Section 4, we describe basic numerical operations on 
tensors represented by TT decomposition such as the addition, Hadamard product, 
matrix-by-vector multiplication, and quadratic form in terms of the 
multilinear operations and TT representations. 
Discussion and conclusions are given in Section 5. 




\section{Notations for tensors and tensor operations}

The notations in this paper follow the convention provided by \cite{Cic2009,KolBa2009}. 
Table \ref{Table:notation_ten} summarizes the notations for tensors. 
Scalars, vectors, and matrices are denoted by lowercase, lowercase bold, and 
uppercase bold letters $x$, $\BF{x}$, and $\BF{X}$, respectively. 
Tensors are denoted by underlined uppercase bold letters $\ten{X}$. 
The $(i_1,i_2,\ldots,i_N)$th entry of a tensor $\ten{X}$ of 
size $I_1\times I_2\times\cdots\times I_N$ is denoted by $x_{i_1,i_2,\ldots,i_N}$
or $\ten{X}(i_1,i_2,\ldots,i_N)$. 
A subtensor of $\ten{X}$ obtained by fixing the indices $i_3,i_4,\ldots,i_N$ is denoted by
$\BF{X}_{:,:,i_3,i_4,\ldots,i_N}$ or $\ten{X}(:,:,i_3,i_4,\ldots,i_N)$. 
We may omit `:' as $\BF{X}_{i_3,i_4,\ldots,i_N}$ if the omitted indices are clear to readers. 
We define the multi-index by 
	$
	\overline{i_1i_2\cdots i_N} = i_N + (i_{N-1}-1)I_N
	+ \cdots + (i_1-1)I_2I_3\cdots I_N
	$
for the indices $i_n=1,2,\ldots,I_n,$ $n=1,2,\ldots,N$. By using this notation, we 
can write an entry of a  Kronecker product as 
$(\BF{a}\otimes \BF{b})_{\overline{ij}} = a_i b_j$. 

Moreover, it is important to note that in this paper the vectorization 
and matricization are defined in accordance with 
the multi-index notation. That is, for a tensor 
$\ten{X}\in\BB{R}^{I_1\times I_2\times\cdots\times I_N}$, 
let $\text{vec}( \ten{X} ) \in\BB{R}^{I_1I_2\cdots I_N}$, 
$\BF{X}_{(n)}\in\BB{R}^{I_n\times I_1I_2\cdots I_{n-1}I_{n+1}\cdots I_N}$, and 
$\BF{X}_{\unfoldto{n}}\in\BB{R}^{I_1I_2\cdots I_n\times I_{n+1}\cdots I_N}$ denote 
the vectorization, mode-$n$ matricization, and mode-$(1,2,\ldots,n)$ matricization of 
$\ten{X}$,
whose entries are defined by 
	\begin{equation} \label{def_matricization}
	\begin{split}
	\left( \text{vec} ( \ten{X} \right) )_{\overline{i_1i_2\cdots i_N} }
	= 
	\left( \BF{X}_{(n)} \right)_{i_n,\overline{i_1i_2\cdots i_{n-1}i_{n+1}\cdots i_N}}
	= 
	\left( \BF{X}_{\unfoldto{n}} \right)_{\overline{i_1i_2\cdots i_n},\overline{i_{n+1}\cdots i_N}}
	= 
	\ten{X}(i_1,i_2,\ldots,i_N)
	\end{split}
	\end{equation}
for $n=1,\ldots,N$.

\begin{table} 
\centering
\caption{\label{Table:notation_ten}Notations for tensors}
\vspace{.5pc}
\begin{tabular}{ll}
\hline\\[-0.8pc]
Notation & Description\\
\hline\\[-0.8pc]
$\ten{X}\in\BB{R}^{I_1\times I_2\times\cdots\times I_N}$
	& $N$th-order tensor of size $I_1\times I_2\times\cdots\times I_N$\\
$x,\BF{x},\BF{X}$ 
	& scalar, vector, and matrix\\
$x_{i_1,i_2,\ldots,i_N}$, $\ten{X}(i_1,i_2,\ldots,i_N)$
	& $(i_1,i_2,\ldots,i_N)$th entry of $\ten{X}$\\
$\BF{x}_{:,i_2,i_3,\ldots,i_N}$, $\ten{X}(:,i_2,i_3,\ldots,i_N)$
	& mode-1 fiber of $\ten{X}$\\
$\BF{X}_{:,:,i_3,i_4,\ldots,i_N}$, $\ten{X}(:,:,i_3,i_4,\ldots,i_N)$
	& frontal slice of $\ten{X}$\\
$\BF{X}_{(n)}\in\BB{R}^{I_n\times I_1I_2\cdots I_{n-1}I_{n+1}\cdots I_N}$
	& mode-$n$ unfolding of $\ten{X}$\\
$\BF{X}_{\unfoldto{n}}\in\BB{R}^{I_1I_2\cdots I_n\times I_{n+1}\cdots I_N}$ 
	& mode-$(1,2,\ldots,n)$ unfolding of $\ten{X}$\\
$\ten{G}, \ten{G}^{(n)}, \ten{X}^{(n)}, \ten{A}^{(n)}$ 
	& core tensors and factor matrices/tensors\\
$R$, $R_n$  & ranks\\
$\overline{i_1i_2\cdots i_N}$  
	& multi-index, $i_N + (i_{N-1}-1)I_N + \cdots + (i_1-1)I_2I_3\cdots I_N$ \\
\hline
\end{tabular}
\end{table}

Table \ref{Table:notation_op} summarizes the notations 
and definitions for basic tensor operations used in this paper. 

\begin{table} 
\centering
\caption{\label{Table:notation_op}
Notations and definitions for basic tensor operations}
\vspace{.5pc}
\begin{tabular}{p{0.2\textwidth}p{0.75\textwidth}}
\hline\\[-0.8pc]
Notation & Description \\
\hline\\[-0.8pc]
$\ten{C}=\ten{A}\otimes\ten{B}$ 
	& Kronecker product of $\ten{A}\in\BB{R}^{I_1\times\cdots\times I_N}$ and 
	    $\ten{B}\in\BB{R}^{J_1\times\cdots\times J_N}$ yields  
	a tensor $\ten{C}$ of size $I_1J_1\times\cdots\times I_NJ_N$ with entries 
	$\ten{C}(\overline{i_1j_1},\ldots,\overline{i_Nj_N}) = \ten{A}(i_1,\ldots,i_N)\ten{B}(j_1,\ldots,j_N)$ \\
$\ten{C}=\ten{A}\circledast\ten{B}$ 
	& Hadamard (elementwise) product of $\ten{A}\in\BB{R}^{I_1\times\cdots\times I_N}$ and 
	    $\ten{B}\in\BB{R}^{I_1\times\cdots\times I_N}$ 
	yields a tensor $\ten{C}$ of size $I_1\times\cdots\times I_N$ with entries 
	$\ten{C}(i_1,\ldots,i_N) = \ten{A}(i_1,\ldots,i_N)\ten{B}(i_1,\ldots,i_N)$\\
$\ten{C}=\ten{A}\circ\ten{B}$ 
	& Outer product of $\ten{A}\in\BB{R}^{I_1\times\cdots\times I_M}$ and 
	$\ten{B}\in\BB{R}^{J_1\times\cdots\times J_N}$ yields 
	a tensor $\ten{C}$ of size $I_1\times\cdots\times I_M\times J_1\times\cdots \times J_N$ with entries 
	$\ten{C}(i_1,\ldots,i_M,j_1,\ldots,j_N) = \ten{A}(i_1,\ldots,i_M)\ten{B}(j_1,\ldots,j_N)$\\
$\ten{C}=\ten{A}\oplus\ten{B}$ 
	& Direct sum of $\ten{A}\in\BB{R}^{I_1\times\cdots\times I_N}$ and 
	$\ten{B}\in\BB{R}^{J_1\times\cdots\times J_N}$ yields
	a tensor $\ten{C}$ of size $(I_1+J_1)\times\cdots\times (I_N+J_N)$ with entries
	$\ten{C}(k_1,\ldots,k_N) = \ten{A}(k_1,\ldots,k_N)$ if $1\leq k_n \leq I_n$ $\forall n$, 
	$\ten{C}(k_1,\ldots,k_N) = \ten{B}(k_1-I_1,\ldots,k_N-I_N)$ if $I_n< k_n \leq I_n+J_n$ $\forall n$,
	and $\ten{C}(k_1,\ldots,k_N) = 0$ otherwise \\
$\ten{C}=\ten{A} \boxtimes \ten{B}$ 
	& Partial Kronecker product of factor tensors 
	   $\ten{A} \in \BB{R}^{R_1\times\cdots\times R_M\times I_1\times\cdots\times I_N}$
	and $\ten{B} \in \BB{R}^{S_1\times\cdots\times S_M\times I_1\times\cdots\times I_N}$ 
	yields a tensor $\ten{C}$ of size
	$R_1S_1 \times\cdots\times R_MS_M \times I_1\times\cdots\times I_N$ with subtensors
	$\ten{C}(:,\ldots,:,i_1,\ldots,i_N) = \ten{A}(:,\ldots,:,i_1,\ldots,i_N) \otimes \ten{B}(:,\ldots,:,i_1,\ldots,i_N)$ \\
$\ten{C}=\ten{A} \boxplus \ten{B}$ 
	& Partial direct sum of factor tensors 
	   $\ten{A} \in \BB{R}^{R_1\times\cdots\times R_M\times I_1\times\cdots\times I_N}$ and
	$\ten{B} \in \BB{R}^{S_1\times\cdots\times S_M\times I_1\times\cdots\times I_N}$ 
	yields a tensor $\ten{C}$ of size 
	$(R_1+S_1) \times\cdots\times (R_M+S_M) \times I_1\times\cdots\times I_N$ with subtensors 
	$\ten{C}(:,\ldots,:,i_1,\ldots,i_N) = \ten{A}(:,\ldots,:,i_1,\ldots,i_N) \oplus \ten{B}(:,\ldots,:,i_1,\ldots,i_N)$ \\
$\ten{C} = \ten{A}\times_n\BF{B}$ 
	& Mode-$n$ product 
		of tensor $\ten{A}\in\BB{R}^{I_1\times\cdots\times I_N}$ and 
		matrix $\BF{B}\in\BB{R}^{J\times I_n}$ 
	yields a tensor $\ten{C}$ of size $I_1\times\cdots\times I_{n-1}\times J\times I_{n+1}\times\cdots \times I_N$ with
	mode-$n$ fibers $\ten{C}(i_1,\ldots,i_{n-1},:,i_{n+1},\ldots,i_N)$
	$= \BF{B}\ten{A}(i_1,\ldots,i_{n-1},:,i_{n+1},\ldots,i_N)$ \\
$\ten{C} = \ten{A}\,\overline{\times}_n\,\BF{b}$ 
	& Mode-$n$ (vector) product 
	of tensor $\ten{A}\in\BB{R}^{I_1\times\cdots\times I_N}$ and 
	vector
	$\BF{b}\in\BB{R}^{I_n}$ yields a tensor $\ten{C}$ of size 
		$I_1\times\cdots\times I_{n-1}\times I_{n+1}\times\cdots \times I_N$ 
	with entries  $\ten{C}(i_1,\ldots,i_{n-1},i_{n+1},\ldots,i_N) 
		= $
	$\BF{b}^\Trps\ten{A}(i_1,\ldots,i_{n-1},:,i_{n+1},\ldots,i_N)$\\
$\left\llbracket \ten{G}; \ten{A}^{(1)},\ldots,\ten{A}^{(N)} \right\rrbracket$
	& Multilinear operator for tensors $\ten{G}$ and $\ten{A}^{(n)}$, $n=1,\ldots,N$, 
		defined by \eqref{def:multiop}\\
$\ten{C} = \ten{A}\modcon\ten{B}$ 
	& Mode-$(M,1)$ contracted product of 
		tensors $\ten{A}\in\BB{R}^{I_1\times\cdots\times I_M}$ and 
	$\ten{B}\in\BB{R}^{J_1\times J_2\times J_3\times\cdots\times J_N}$ 
	   with $I_M=J_1$ yields 
	a tensor $\ten{C}$ of size $I_1\times\cdots\times I_{M-1}\times J_2\times\cdots \times J_N$ 
		with entries 
	$\ten{C}(i_1,\ldots,i_{M-1},j_2,\ldots,j_N) = 
		\sum_{i_M=1}^{I_M} \ten{A}(i_1,\ldots,i_M) \ten{B}(i_M,j_2,\ldots,j_N)$\\
$\BF{C}=\BF{A} \skron \BF{B}$
	& Strong Kronecker product of two block matrices 
	$\BF{A}=[\BF{A}_{r_1,r_2}]\in\BB{R}^{R_1I_1\times R_2I_2}$ and $\BF{B}=[\BF{B}_{r_2,r_3}]\in\BB{R}^{R_2J_1\times R_3J_2}$ yields
	a block matrix $\BF{C}=[\BF{C}_{r_1,r_3}]\in\BB{R}^{R_1I_1J_1\times R_3I_2J_2}$ with blocks
	$\BF{C}_{r_1,r_3} = \sum_{r_2=1}^{R_2}\BF{A}_{r_1,r_2}\otimes
		\BF{B}_{r_2,r_3}$ \\
$\ten{C}=\ten{A} \skron \ten{B}$
	&Strong Kronecker product of two block tensors 
	$\ten{A}=[\ten{A}_{r_1,r_2}]\in\BB{R}^{R_1I_1\times R_2I_2 \times I_3}$
		and $\ten{B}=[\ten{B}_{r_2,r_3}]\in\BB{R}^{R_2J_1\times R_3J_2\times J_3}$ yields
	a block tensor $\ten{C}=[\ten{C}_{r_1,r_3}]\in\BB{R}^{R_1I_1J_1\times R_3I_2J_2\times I_3J_3}$ 
		with blocks 
	$\ten{C}_{r_1,r_3} = \sum_{r_2=1}^{R_2}\ten{A}_{r_1,r_2}\otimes
		\ten{B}_{r_2,r_3}$\\
$\ptrace\left( \ten{X} \right)$ 
	& Partial trace operator $\ptrace:\BB{R}^{R\times I_1\times \cdots \times I_{N}\times R}
	\rightarrow \BB{R}^{I_1\times I_2\times\cdots\times I_{N}}$
	is 
	defined  by  $\ptrace(\ten{X}) = \sum_{r=1}^{R} \ten{X}\left(r,i_1,\ldots,i_{N},r \right)$
	\\
\hline
\end{tabular}
\end{table}

\subsection{Graphical representations of tensors}

It is quite useful to visualize tensors and related operations by 
tensor network diagrams, e.g., see \cite{Holtz2012}, and references in \cite{Gra2013}. 
Figure \ref{Fig:TensorGraph}(a), (b), and (c) illustrate
tensor network diagrams representing a vector, a matrix, and a 3rd-order tensor. 
In each graph, the number of edges connected to a node
indicates the order of the tensor, and the mode size can be shown 
by the label on each edge. 
Figure \ref{Fig:TensorGraph}(d) represents the singular 
value decomposition of a matrix. 
The orthonormalized matrices are represented by 
half-filled circles and the diagonal matrix by a circle with slash inside. 
Figure \ref{Fig:TensorGraph}(e) represents the
mode-3 product, $\ten{A}\times_3\BF{B}$, of a tensor 
$\ten{A}\in\BB{R}^{I_1\times I_2\times I_3}$ with a matrix 
$\BF{B}\in\BB{R}^{J_1\times J_2}$ $(I_3=J_2)$. 
Figure \ref{Fig:TensorGraph}(f) represents the
contracted product, $\ten{A}\modcon\ten{B}$, of a tensor 
$\ten{A}\in\BB{R}^{I_1\times I_2\times I_3}$ with a tensor 
$\ten{B}\in\BB{R}^{J_1\times J_2\times J_3}$ $(I_3=J_1)$. 

\begin{figure}
\centering
\begin{tabular}{ccc}
\includegraphics[width=1.1cm]{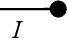}& 
\includegraphics[width=2cm]{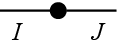}& 
\includegraphics[width=2cm]{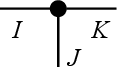}
\\
(a) & (b) & (c)
\\
\includegraphics[width=4.0cm]{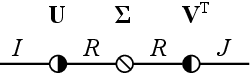}&
\includegraphics[width=3.0cm]{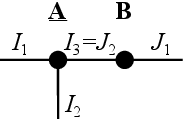}&
\includegraphics[width=3.0cm]{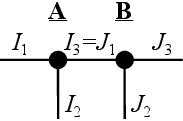}
\\
(d)&(e)&(f)
\end{tabular}
\caption{\label{Fig:TensorGraph}
Graphical representations for (a) a vector (b) a matrix, (c) a 3rd-order tensor, 
(d) singular value decomposition of 
an $I\times J$ matrix, 
(e) mode-$3$ product of a 3rd-order tensor with a matrix,
and (f) contracted product of two 3rd-order tensors.}
\end{figure}

\subsection{Kronecker, Hadamard, and outer products}

Definitions for traditional matrix-matrix product operations
such as the Kronecker, Hadamard, and outer products 
can immediately be generalized to tensor-tensor products. 

\begin{defn}[Kronecker product]
The Kronecker product of tensors $\ten{A}\in\BB{R}^{I_1\times I_2\times \cdots \times I_N}$
and $\ten{B}\in\BB{R}^{J_1\times J_2\times \cdots \times J_N}$ is defined by 
	$$
	\ten{C}=\ten{A}\otimes\ten{B}\in\BB{R}^{I_1J_1\times I_2J_2\times \cdots \times I_NJ_N}
	$$
with entries 
	$$
	\ten{C}({\overline{i_1j_1},\overline{i_2j_2},\ldots,\overline{i_Nj_N} })
	= \ten{A}({i_1,i_2,\ldots,i_N}) \ten{B}({j_1,j_2,\ldots,j_N}) 
	$$
for $i_n=1,2,\ldots,I_n,$ $j_n=1,2,\ldots,J_n,$ $n=1,\ldots,N$. 

\end{defn}

\begin{defn}[Hadamard product]
The Hadamard (elementwise) product of $\ten{A}\in\BB{R}^{I_1\times I_2\times\cdots\times I_N}$
and $\ten{B}\in\BB{R}^{I_1\times I_2\times\cdots\times I_N}$ is defined by 
	$$
	\ten{C}=\ten{A}\circledast\ten{B}
	\in\BB{R}^{I_1\times I_2\times\cdots\times I_N}
	$$
with entries 
	$$
	\ten{C}(i_1,i_2,\ldots,i_N) = 
	\ten{A}(i_1,i_2,\ldots,i_N) 
	\ten{B}(i_1,i_2,\ldots,i_N)
	$$
for $i_n=1,2,\ldots,I_n,$ $n=1,\ldots,N$. 
\end{defn}
\begin{defn}[Outer product]
The outer product of $\ten{A}\in\BB{R}^{I_1\times I_2\times\cdots\times I_M}$
and $\ten{B}\in\BB{R}^{J_1\times J_2\times\cdots\times J_N}$ is defined by 
	$$
	\ten{C}=\ten{A}\circ\ten{B}
	\in\BB{R}^{I_1\times I_2\times\cdots\times I_M
		\times J_1\times J_2\times \cdots \times J_N}
	$$
with entries 
	$$
	\ten{C}(i_1,i_2,\ldots,i_M,j_1,j_2,\ldots,j_N) = 
	\ten{A}(i_1,i_2,\ldots,i_M) 
	\ten{B}(j_1,j_2,\ldots,j_N)
	$$
for $i_m=1,2,\ldots,I_m,$ $m=1,\ldots,M$, $j_n=1,2,\ldots,J_n,$ $n=1,\ldots,N$. 
\end{defn}

Note that an $N$th-order tensor $\ten{X}\in\BB{R}^{I_1\times I_2\times\cdots\times I_N}$
is rank-one if it is written as the outer product of $N$ vectors
	$$
	\ten{X} = \BF{x}^{(1)}\circ \BF{x}^{(2)}\circ\cdots\circ\BF{x}^{(N)}. 
	$$
In general, $N$th-order tensor $\ten{X}\in\BB{R}^{I_1\times I_2\times \cdots\times I_N}$
can be represented as a sum of rank-one tensors as (so called CP or PARAFAC \cite{KolBa2009})
	$$
	\ten{X} = \sum_{r=1}^R \BF{x}^{(1)}_r\circ \BF{x}^{(2)}_r
	\circ\cdots\circ\BF{x}^{(N)}_r. 
	$$
The smallest number $R$ of the rank-one tensors that produce $\ten{X}$
is called the tensor rank of $\ten{X}$ \cite{KolBa2009}. 
We can define a tensor operation between rank-one tensors and generalize it
to sums of rank-one tensors. For example, let $\ten{A}=\BF{a}^{(1)}\circ \BF{a}^{(2)}\circ\cdots\circ\BF{a}^{(N)}$ 
and $\ten{B}=\BF{b}^{(1)}\circ \BF{b}^{(2)}\circ\cdots\circ\BF{b}^{(N)}$ 
denote two rank-one tensors, and let 
$\BF{a} \circledast \BF{b} = (a_i b_i)$ denote Hadamard (elementwise) product of vectors, then, 
\begin{itemize}
\item the Kronecker product $\ten{A}\otimes\ten{B}$ can be expressed by
	$$
	\ten{A} \otimes \ten{B}
	= 
	\left(\BF{a}^{(1)}\otimes\BF{b}^{(1)}\right)
	\circ \left(\BF{a}^{(2)}\otimes\BF{b}^{(2)}\right)
	\circ\cdots
	\circ\left(\BF{a}^{(N)}\otimes\BF{b}^{(N)}\right), 
	$$
\item the Hadamard product by 
	$$
	\ten{A} \circledast \ten{B}
	= 
	\left(\BF{a}^{(1)}\circledast\BF{b}^{(1)}\right)
	\circ \left(\BF{a}^{(2)}\circledast\BF{b}^{(2)}\right)
	\circ\cdots
	\circ\left(\BF{a}^{(N)}\circledast\BF{b}^{(N)}\right), 
	$$
\item and the outer product by 
	$$
	\ten{A} \circ \ten{B}
	= 
	\BF{a}^{(1)} \circ\cdots\circ \BF{a}^{(N)}\circ
	\BF{b}^{(1)} \circ\cdots\circ \BF{b}^{(N)}. 
	$$
\end{itemize}
However, the problem of determining the tensor rank of a 
specific tensor is NP-hard in general if the order 
is larger than 2 \cite{Hastad90}. So, for practical applications, 
we will define tensor operations 
by using index notation and provide examples with 
rank-one tensors.

\subsection{Direct sum}

The direct sum of matrices $\BF{A}$ and $\BF{B}$ is defined by 
	$$
	\BF{A}\oplus\BF{B} = \text{diag}\left(\BF{A},\BF{B}\right)
	= \begin{bmatrix}\BF{A} & \BF{0} \\ \BF{0} & \BF{B}
	\end{bmatrix}. 
	$$
A generalization of the direct sum to tensors is defined as follows. 
\begin{defn}[Direct sum]
The direct sum of tensors $\ten{A}\in\BB{R}^{I_1\times I_2\times\cdots\times I_N}$
and $\ten{B}\in\BB{R}^{J_1\times J_2\times\cdots\times J_N}$ is defined by 
	$$
	\ten{C}=\ten{A}\oplus\ten{B}
	\in\BB{R}^{(I_1+J_1)\times(I_2+J_2)\times\cdots\times (I_N+J_N)}
	$$
with entries 
	$$
	\ten{C}(k_1,k_2,\ldots,k_N) = 
	\begin{cases}
	\ten{A}(k_1,k_2,\ldots,k_N) & \text{if }1\leq k_n \leq I_n \ \forall n\\
	\ten{B}(k_1-I_1,k_2-I_2,\ldots,k_N-I_N) & \text{if }I_n<k_n\leq I_n+J_n\ \forall n\\
	0 & \text{otherwise}.
	\end{cases}
	$$

%

\end{defn}

As special cases, the direct sum of vectors $\BF{a}\in\BB{R}^{I}$ and $\BF{b}\in\BB{R}^{J}$
is the concatenated vector $\BF{a}\oplus\BF{b}\in\BB{R}^{I+J}$, and 
the direct sum of matrices $\BF{A}\in\BB{R}^{I_1\times I_2}$ and 
$\BF{B}\in\BB{R}^{J_1\times J_2}$ is the block diagonal matrix
$\BF{A}\oplus\BF{B}=\text{diag}(\BF{A},\BF{B}) \in\BB{R}^{(I_1+J_1)\times (J_1+J_2)}$. 
We suppose that the direct sum of scalars $a,b\in\BB{R}$ is the addition 
$a\oplus b=a+b\in\BB{R}$. 
Similarly, direct sum of two 3rd-order tensors is a block diagonal 3rd-order tensor, 
as illustrated in Figure~\ref{Fig_direct_sum_blocks}. 

\begin{figure}
\centering
\begin{tabular}{c}
\includegraphics[width=2.5cm]{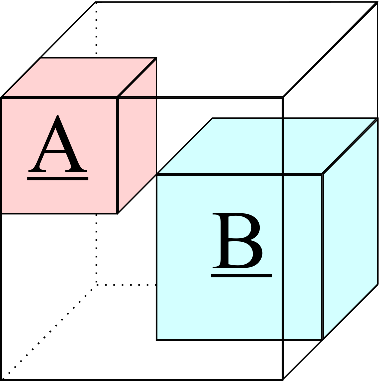}
\end{tabular}
\caption{\label{Fig_direct_sum_blocks}
Illustration of direct sum $\ten{A}\oplus\ten{B}$ of 3rd-order tenosors 
$\ten{A}$ and $\ten{B}$. }
\end{figure}

\subsection{Partial Kronecker product and partial direct sum: 
General operations for factor tensors}

A low-rank tensor decomposition method approximately represents a tensor as contraction of 
a collection of factor (core) tensors, which helps to reduce the number of representation parameters. 
A {\em factor tensor} refers to a tensor which forms such a collection 
in a low-rank tensor representation \cite{EspHacHanSch2011}. 
In general, each mode of a factor tensor can be classified as 
either a physical (spatial) mode or an auxiliary mode \cite{Cirac2013,EspNarSch2012}. 
In this paper, sizes of auxiliary modes are denoted by $R$ or $S$, and the corresponding 
indices are denoted by $r$ or $s$. For example, if a factor tensor $\ten{X}$ is mentioned to have 
size $R_1\times I_1 \times R_2 \times I_2$, then it implies that the modes 1 and 3 are auxiliary modes 
and the rest are physical modes. A factor tensor may have both types of or only one type of the physical and auxiliary modes. 

The partial Kronecker product and the partial direct sum, which will  be defined below, are 
generalizations of the Kronecker product and the direct sum to factor tensors. 

\begin{defn}[Partial Kronecker product and partial direct sum]
Let $\ten{A}\in\BB{R}^{R_1\times R_2\times \cdots \times R_M \times I_1\times\cdots\times I_N}$
and $\ten{B}\in\BB{R}^{S_1\times S_2\times \cdots \times S_M \times I_1\times\cdots\times I_N}$
be two factor tensors with $M$ auxiliary modes and $N$ physical modes. 
The partial Kronecker product of $\ten{A}$ and $\ten{B}$ is defined by 
	$$
	\ten{C} = \ten{A}\boxtimes\ten{B}
	\in \BB{R}^{R_1S_1\times R_2S_2\times \cdots \times R_MS_M \times I_1\times\cdots\times I_N}
	$$
with 	subtensors 
	$$
	\ten{C}(:,\ldots,:,i_1,\ldots,i_N) = \ten{A}(:,\ldots,:,i_1,\ldots,i_N) \otimes \ten{B}(:,:,\ldots,:,i_1,\ldots,i_N)
	$$
for $i_n=1,2,\ldots,I_n,$ $n=1,\ldots,N$. 
Similarly, the partial direct sum of $\ten{A}$ and $\ten{B}$ is defined by 
	$$
	\ten{C} = \ten{A}\boxplus\ten{B}
	\in \BB{R}^{(R_1+S_1)\times (R_2+S_2) \times \cdots \times (R_M+S_M) \times I_1\times\cdots\times I_N}
	$$
with 	subtensors
	$$
	\ten{C}(:,\ldots,:,i_1,\ldots,i_N) = \ten{A}(:,\ldots,:,i_1,\ldots,i_N) \oplus \ten{B}(:,:,\ldots,:,i_1,\ldots,i_N)
	$$
for $i_n=1,2,\ldots,I_n,$ $n=1,\ldots,N$. 
\end{defn}

In the above definition, if $M=0$, i.e., the tensors $\ten{A}$ and $\ten{B}$ have only physical modes, 
then the partial Kronecker product and the partial direct sum are equivalent to 
the Hadamard (elementwise) product and the elementwise addition, respectively. 
On the other hand, if $N=0$, i.e., there are only auxiliary modes, 
then the partial Kronecker product and the partial direct sum are equivalent to 
the Kronecker product and the direct sum, respectively.

\subsection{Multilinear operator}

The mode-$n$ product of a tensor $\ten{G}\in\BB{R}^{R_1\times R_2\times\cdots\times R_N}$
and a matrix $\BF{A}\in\BB{R}^{I_n\times R_n}$ is a multilinear operator defined by \cite{KolBa2009}
	\begin{equation} \label{mode_n_matrix}
	\ten{X} = \ten{G}\times_n\BF{A}\in\BB{R}^{R_1\times R_2\times\cdots\times 
			R_{n-1}\times I_n\times R_{n+1}\times\cdots\times R_N}
	\end{equation}
with entries 
	$$
	\ten{X}(r_1,r_2,\ldots,r_{n-1},i_n,r_{n+1},\ldots,r_N)
	= \sum_{r_n=1}^{R_n}
	\ten{G}(r_1,r_2,\ldots,r_N) \BF{A}(i_n,r_n). 
	$$
A few selected properties of the mode-$n$ product are listed as follows\footnote{
	In Proposition~\ref{prop_mode_n_mat}(e) and (f), 
	factors of Kronecker products are in a different (reversed) order, e.g., 
	increasing from 1 to $N$, compared to the order in the literature
	 \cite{DelathMoorVan2000,Kol2006,KolBa2009}, e.g., 
	 decreasing from $N$ to 1, due to the definition of matricizations in 
	\eqref{def_matricization}.
	}.

\begin{prop}[\cite{DelathMoorVan2000,Kol2006,KolBa2009}]
\label{prop_mode_n_mat}

	Let $\ten{G}\in\BB{R}^{R_1\times R_2\times\cdots\times R_N}$ be an $N$th-order tensor. Then
	\begin{enumerate}
	\item[(a)]
		$\ten{G}\times_m \BF{A} \times_n\BF{B} = \ten{G}\times_n\BF{B}\times_m\BF{A}$
		for $m\neq n$. 
	\item[(b)] 
		$\ten{G}\times_n \BF{A} \times_n\BF{B} = \ten{G}\times_n\BF{BA}$. 
	\item[(c)]
		If $\BF{A}$ has full column rank, then 
		$$
		\ten{X} = \ten{G}\times_n \BF{A}
		\Rightarrow \ten{G} = \ten{X} \times_n \BF{A}^{\dagger},
		$$
		where $\BF{A}^\dagger$ is the Moore-Penrose pseudoinverse of $\BF{A}$. 
		
	\item[(d)]
		If $\BF{A}\in\BB{R}^{I \times R_n}$, then 
		$$
		\ten{X} = \ten{G}\times_n\BF{A}
		\Leftrightarrow
		\BF{X}_{(n)} = \BF{A}\BF{G}_{(n)}.$$
		
	\item[(e)]
		If $\BF{A}^{(n)}\in\BB{R}^{I_n\times R_n}$ for all $n=1,\ldots,N$, we have, for $n\in\{1,\ldots,N\}$, 
		\begin{multline*}
		\ten{X} = \ten{G}\times_1 \BF{A}^{(1)} \times_2 \BF{A}^{(2)} \cdots \times_N \BF{A}^{(N)}
				\Leftrightarrow \\
		\BF{X}_{(n)} = \BF{A}^{(n)} \BF{G}_{(n)} \left( \BF{A}^{(1)} \otimes \cdots \otimes 
		\BF{A}^{(n-1)} \otimes \BF{A}^{(n+1)} \otimes \cdots \otimes \BF{A}^{(N)} \right)^\Trps. 
		\end{multline*}
		
	\item[(f)]
		If $\BF{A}^{(n)}\in\BB{R}^{I_n\times R_n}$ for all $n=1,\ldots,N$, we have, for $n\in\{1,\ldots,N\}$, 
		\begin{multline*}
		\ten{X} = \ten{G}\times_1 \BF{A}^{(1)} \times_2 \BF{A}^{(2)} \cdots \times_N \BF{A}^{(N)}
			\Leftrightarrow \\
		\BF{X}_{\unfoldto{n}} =  \left( \BF{A}^{(1)} \otimes \cdots \otimes \BF{A}^{(n)} \right)
			\BF{G}_{\unfoldto{n}} \left( \BF{A}^{(n+1)} \otimes \cdots \otimes \BF{A}^{(N)} \right)^\Trps. 
		\end{multline*}		
	\end{enumerate}
\end{prop}

Kolda and Bader \cite{KolBa2009} further introduced a multilinear operator 
called the Tucker operator \cite{Kol2006} to simplify the expression for the mode-$n$ product. 
The Tucker operator of a tensor $\ten{G}\in\BB{R}^{R_1\times R_2\times \cdots \times R_N}$
and matrices $\BF{A}^{(n)}\in\BB{R}^{I_n\times R_n},n=1,\ldots,N,$ is
defined by 
	\begin{equation}\label{eqn:multiopMatrix}
	\left\llbracket \ten{G} ; \BF{A}^{(1)}, \BF{A}^{(2)}, \ldots, \BF{A}^{(N)}
	\right\rrbracket 
	= \ten{G} \times_1\BF{A}^{(1)} \times_2\BF{A}^{(2)} \times_3\cdots \times_N\BF{A}^{(N)}
	\in\BB{R}^{I_1\times I_2\times\cdots\times I_N}. 
	\end{equation}
Here, we generalize this to a multilinear operator
between tensors. 

\begin{defn}[Multilinear operator]
Let $N \geq 1$ and $M_n\geq 0$ for $n=1,\ldots,N$. 
For an $N$th-order tensor $\ten{G}\in\BB{R}^{R_1\times R_2\times\cdots\times R_N}$
and $(M_n+1)$th-order tensors 
$\ten{A}^{(n)}\in\BB{R}^{I_{n,1}\times I_{n,2}\times\cdots\times 
I_{n,M_n}\times R_n}$, $n=1,\ldots,N,$
the multilinear operator is defined by the $(M_1+M_2+\cdots+M_N)$th-order tensor
	\begin{equation} \label{def:multiop}
	\ten{X} = \left\llbracket \ten{G}; \ten{A}^{(1)},\ldots,
	\ten{A}^{(N)}\right\rrbracket
	\in\BB{R}^{I_{1,1}\times\cdots\times 
	I_{1,M_1} \times\cdots\times
	I_{N,1}\times\cdots\times 
	I_{N,M_N}}
	\end{equation}
with entries
	\begin{equation*}
	\ten{X}(\BF{i}_1,\BF{i}_2,\ldots,\BF{i}_N) = 
	\sum_{r_1=1}^{R_1}\sum_{r_2=1}^{R_2}\cdots\sum_{r_N=1}^{R_N}
	\ten{G}(r_1,r_2,\ldots,r_N)
	\ten{A}^{(1)}(\BF{i}_1,r_1) \ten{A}^{(2)}(\BF{i}_2,r_2) \cdots 
	\ten{A}^{(N)}(\BF{i}_N,r_N),
	\end{equation*}
where $\BF{i}_n=(i_{n,1},i_{n,2},\ldots,i_{n,M_n})$ is the ordered indices. 

\end{defn}

Figure~\ref{Fig_multilinear}(a) illustrates the tensor network diagram for 
multilinear operator $\llbracket \ten{G}; \ten{A}^{(1)},\ldots,\ten{A}^{(N)} \rrbracket$ 
with an $N$th-order tensor $\ten{G}$ and 4th-order tensors $\ten{A}^{(1)}$, \ldots, $\ten{A}^{(N)}$. 

As a special case, if $\ten{A}^{(n)}$ are matrices, i.e., $M_n=1$ for all $n$, then 
the multilinear operator \eqref{def:multiop} is equivalent to 
the standard Tucker operator \eqref{eqn:multiopMatrix}. 
Moreover, in the case of vectors $\BF{a}^{(n)}\in\BB{R}^{R_n}$, 
i.e., $M_n=0$, we have the scalar
	$$
	\left\llbracket \ten{G}; \BF{a}^{(1)}, \BF{a}^{(2)}, \ldots, \BF{a}^{(N)}
	\right\rrbracket
	=
	\ten{G} \,\overline{\times}_1\, \BF{a}^{(1)} \,\overline{\times}_2\, \BF{a}^{(2)}
	\,\overline{\times}_3\,
	\cdots \,\overline{\times}_N\, \BF{a}^{(N)}
	\in\BB{R},
	$$
where $\overline{\times}_n$ is the mode-$n$ (vector)
product \cite{KolBa2009}. 

\begin{exmp}
\label{example_multilinear_rank1}
Let $\ten{G}=\BF{g}^{(1)}\circ\BF{g}^{(2)}\circ\BF{g}^{(3)}\in
\BB{R}^{R_1\times R_2\times R_3}$
and 
$\ten{A}=\BF{a}^{(1)}\circ\BF{a}^{(2)}\circ
\BF{a}^{(3)}\circ\BF{a}^{(4)}\circ\BF{a}^{(5)}\in
\BB{R}^{I_1\times I_2\times I_3\times I_4\times R_2}$ be 
rank-one tensors. 
Then, 
	$$
	\left\llbracket
	\ten{G}; \BF{I}_{R_1}, \ten{A}, \BF{I}_{R_3} \right\rrbracket
	= \langle \BF{g}^{(2)}, \BF{a}^{(5)} \rangle
	\cdot 
	\BF{g}^{(1)} \circ \BF{a}^{(1)}\circ\BF{a}^{(2)}\circ\BF{a}^{(3)}\circ\BF{a}^{(4)} 
	\circ \BF{g}^{(3)} 
	\in\BB{R}^{R_1\times I_1\times I_2\times I_3\times I_4 \times R_3}, 
	$$
where $\langle \BF{v}, \BF{w} \rangle = \BF{v}^\Trps\BF{w}$ 
is the innerproduct of vectors, see Figure~\ref{Fig_multilinear}(b).  
\end{exmp}

\begin{figure}
\centering
\begin{tabular}{ccc}
\includegraphics[width=5cm]{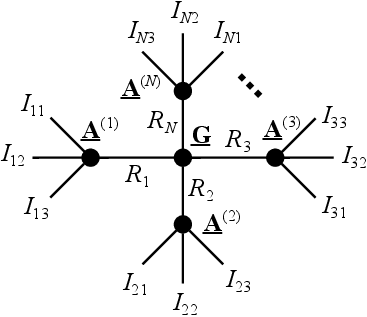}& \hspace{2pc} &
\includegraphics[width=5cm]{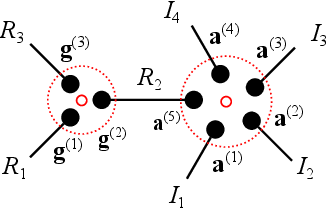}
\\
(a) & & (b)
\end{tabular}
\caption{\label{Fig_multilinear}
Tensor network diagrams for (a) multilinear operator 
$\llbracket \ten{G}; \ten{A}^{(1)},\ldots,\ten{A}^{(N)} \rrbracket$ 
with $N$th-order tensor $\ten{G}$ and  4th-order tensors $\ten{A}^{(1)}$, \ldots, $\ten{A}^{(N)}$, and 
(b) multilinear operator 
$\llbracket \ten{G}; \BF{I}_{R_1}, \ten{A}, \BF{I}_{R_3} \rrbracket$ 
with rank-one tensors $\ten{G}=\BF{g}^{(1)}\circ \BF{g}^{(2)}\circ \BF{g}^{(3)}$ and 
$\ten{A} = \BF{a}^{(1)}\circ \BF{a}^{(2)}\circ \cdots\circ \BF{a}^{(5)}$ 
as described in Example~\ref{example_multilinear_rank1}. 
}
\end{figure}

We can derive the following properties.

\begin{prop} 
Let $N\geq 1$ and $M_n\geq 0$ for $n=1,\ldots,N$. 
Let $\ten{G}_A$ and $\ten{G}_B$ be $N$th-order tensors and 
$\ten{A}^{(n)}$ and $\ten{B}^{(n)}$ be $(M_n+1)$th-order factor tensors  
whose $(M_n+1)$th modes are auxiliary modes for $n=1,\ldots,N$. Let  
	$$
	\ten{A} = \left\llbracket \ten{G}_A; \ten{A}^{(1)},\ldots,\ten{A}^{(N)}
	\right\rrbracket
	\quad
	\text{and}
	\quad
	\ten{B} = \left\llbracket \ten{G}_B; \ten{B}^{(1)},\ldots,\ten{B}^{(N)}
	\right\rrbracket. 
	$$
Then 
	\begin{enumerate}
	\item[(a)] $\ten{A}+\ten{B} = 
			\left\llbracket \ten{G}_A\oplus\ten{G}_B; 
			\ten{A}^{(1)} \boxplus \ten{B}^{(1)},\ldots,
			\ten{A}^{(N)} \boxplus \ten{B}^{(N)}
			\right\rrbracket$ if $\ten{A}$  and $\ten{B}$ have the same size.  
	\item[(b)] $\ten{A}\oplus\ten{B} = 
			\left\llbracket \ten{G}_A\oplus\ten{G}_B; 
			\ten{A}^{(1)} \oplus \ten{B}^{(1)},\ldots,
			\ten{A}^{(N)} \oplus \ten{B}^{(N)}
			\right\rrbracket$. 
	\item[(c)] $\ten{A}\circledast\ten{B} = 
			\left\llbracket \ten{G}_A\otimes\ten{G}_B; 
			\ten{A}^{(1)} \boxtimes \ten{B}^{(1)},\ldots,
			\ten{A}^{(N)} \boxtimes \ten{B}^{(N)}
			\right\rrbracket$ if $\ten{A}$  and $\ten{B}$ have the same size. 
	\item[(d)] $\ten{A}\otimes\ten{B} = 
			\left\llbracket \ten{G}_A\otimes\ten{G}_B; 
			\ten{A}^{(1)} \otimes \ten{B}^{(1)},\ldots,
			\ten{A}^{(N)} \otimes \ten{B}^{(N)}
			\right\rrbracket$. 
	\end{enumerate}
\end{prop}
\begin{proof} (a) to (d) can be derived from the definitions of the corresponding 
operations and algebraic manipulation. 
\end{proof}

\begin{exmp} We consider the examples where 
$\ten{A}^{(n)}$ and $\ten{B}^{(n)}$ are either factor matrices or vectors. 
	\begin{enumerate}
	\item Let $M_n=1$ for $n=1,\ldots,N$, i.e., the tensors $\ten{A}$ and $\ten{B}$ have the form
		(which is the Tucker decomposition, to be introduced in Section~\ref{sec:cp_n_tucker})
			$$
			\ten{A} = \left\llbracket \ten{G}_A; \BF{A}^{(1)},\ldots,\BF{A}^{(N)}
			\right\rrbracket \in\BB{R}^{I_1\times I_2\times\cdots\times I_N}, 
			$$
			$$
			\ten{B} = \left\llbracket \ten{G}_B; \BF{B}^{(1)},\ldots,\BF{B}^{(N)}
			\right\rrbracket \in\BB{R}^{I_1\times I_2\times\cdots\times I_N}. 
			$$
		It follows that the Kronecker product, Hadamard product, direct sum, 
		and addition lead to tensors in the same form: 
			\begin{enumerate}
			\item[(a)] $\ten{A}+\ten{B} = 
					\left\llbracket \ten{G}_A\oplus\ten{G}_B; 
					\BF{A}^{(1)} \boxplus \BF{B}^{(1)},\ldots,
					\BF{A}^{(N)} \boxplus \BF{B}^{(N)}
					\right\rrbracket$. 
			\item[(b)] $\ten{A}\oplus\ten{B} = 
					\left\llbracket \ten{G}_A\oplus\ten{G}_B; 
					\BF{A}^{(1)}\oplus \BF{B}^{(1)},\ldots,
					\BF{A}^{(N)}\oplus\BF{B}^{(N)}
					\right\rrbracket$. 
			\item[(c)] $\ten{A}\circledast\ten{B} = 
					\left\llbracket \ten{G}_A\otimes\ten{G}_B; 
					\BF{A}^{(1)} \boxtimes \BF{B}^{(1)},\ldots,
					\BF{A}^{(N)} \boxtimes \BF{B}^{(N)}
					\right\rrbracket$. 
			\item[(d)] $\ten{A}\otimes\ten{B} = 
					\left\llbracket \ten{G}_A\otimes\ten{G}_B; 
					\BF{A}^{(1)}\otimes \BF{B}^{(1)},\ldots,
					\BF{A}^{(N)}\otimes\BF{B}^{(N)}
					\right\rrbracket$. 
			\end{enumerate}
		Moreover, if the core tensors $\ten{G}_A$ and $\ten{G}_B$ are superdiagonal tensors, 
		which is the case of the CP decomposition (see Section~\ref{sec:cp_n_tucker},)
		then the results are also given as the CP decomposition because the Kronecker 
		product and the direct sum of superdiagonal core tensors 	are superdiagonal 
		tensors as well. 
	\item  Let $M_n=0$ for $n=1,\ldots,N$, then we have the scalars
			$$
			a = \left\llbracket \ten{G}_A; \BF{a}^{(1)},\ldots,\BF{a}^{(N)}
			\right\rrbracket \in\BB{R}, 
			$$
			$$
			b = \left\llbracket \ten{G}_B; \BF{b}^{(1)},\ldots,\BF{b}^{(N)}
			\right\rrbracket \in\BB{R}. 
			$$
		The addition and multiplication are given in the form
			\begin{enumerate}
			\item[(a)] $a+b = a\oplus b =
					\left\llbracket \ten{G}_A\oplus\ten{G}_B; 
					\BF{a}^{(1)}\oplus \BF{b}^{(1)},\ldots,
					\BF{a}^{(N)}\oplus\BF{b}^{(N)}
					\right\rrbracket$, 
			\item[(b)] $ab = a\otimes b = a\circledast b = 
					\left\llbracket \ten{G}_A\otimes\ten{G}_B; 
					\BF{a}^{(1)}\otimes \BF{b}^{(1)},\ldots,
					\BF{a}^{(N)}\otimes\BF{b}^{(N)}
					\right\rrbracket$. 
			\end{enumerate}
	\end{enumerate}
\end{exmp}

\subsection{Contracted product} 
\label{sec_modeM1_contracted}

The mode-$n$ product of a tensor with a matrix in \eqref{mode_n_matrix} 
can be extended to a product between tensors of any orders. 
We define one of the simplest cases of the tensor-by-tensor contracted product
as follows. 

\begin{defn}[Mode-$(M,1)$ contracted product]
Let $M, N \geq 1$. 
The mode-($M,1$) contracted product
of tensors $\ten{A}\in\BB{R}^{I_1\times I_2\times\cdots\times I_M}$ and 
$\ten{B}\in\BB{R}^{J_1\times J_2\times\cdots\times J_N}$ with $I_M=J_1$ is defined by 
	\begin{equation*}
	\ten{C} = \ten{A}\modcon\ten{B} \in\BB{R}^{I_1\times\cdots\times I_{M-1}\times J_2\times\cdots \times J_N}
	\end{equation*}
with entries 
	$$
	\ten{C}(i_1,\ldots,i_{M-1},j_2,\ldots,j_N)
	= \sum_{i_M=1}^{I_M} \ten{A}(i_1,\ldots,i_M)
	\ten{B}(i_M,j_2,\ldots,j_N)
	$$
for all $i_m$, $j_n$, $m=1,\ldots,M-1,$ $n=2,\ldots,N$. 

\end{defn}

We note that the tensor-by-tensor contracted product defined above
is a natural generalization of the matrix multiplication as 
$\BF{A}\BF{B} = \BF{A}\modcon\BF{B}$, and the 
vector innerproduct as 
$\langle \BF{a}, \BF{b} \rangle = \BF{a}\modcon\BF{b}$. 
Especially, the contracted product between a tensor $\ten{A}
\in\BB{R}^{I_1\times \cdots\times I_M}$ and a vector
$\BF{p}\in\BB{R}^{I_1}$ or $\BF{q}\in\BB{R}^{I_M}$
produces a tensor of smaller order as 
	\begin{equation*}
	\begin{split}
	\BF{p}\modcon \ten{A}    & \in\BB{R}^{I_2\times\cdots\times I_M}, \\
	\ten{A}\modcon \BF{q}    & \in\BB{R}^{I_1\times\cdots\times I_{M-1}}.
	\end{split}
	\end{equation*}
\begin{exmp}
The contracted product of rank-one tensors yields
	$$
	\left(\BF{a}^{(1)}\circ\cdots\circ\BF{a}^{(M)} \right)
	\modcon
	\left(\BF{b}^{(1)}\circ\cdots\circ\BF{b}^{(N)} \right)
	= 
	\left\langle \BF{a}^{(M)}, \BF{b}^{(1)} \right\rangle
	\cdot \BF{a}^{(1)}\circ\cdots\circ\BF{a}^{(M-1)}
	\circ
	\BF{b}^{(2)}\circ\cdots\circ\BF{b}^{(N)}. 
	$$
\end{exmp}

In general, we have the following properties: 

\begin{prop} 
\label{prop_modcon}
Let $\ten{A}\in\BB{R}^{I_1\times I_2\times\cdots\times I_M}$, 
$\ten{B}\in\BB{R}^{I_M\times J_2\times\cdots\times J_N}$, 
$\ten{C}\in\BB{R}^{J_N\times K_2\times\cdots\times K_L}$, 
$\ten{G}\in\BB{R}^{R_1\times\cdots\times R_N}$, 
$\BF{P}\in\BB{R}^{I\times R_1}$, and 
$\BF{Q}\in\BB{R}^{R_N \times J}$. 
Then, 
	\begin{enumerate}
	\item[(a)] 
	$\ten{A}\modcon\ten{B} = \llbracket \ten{B}; \ten{A}, \BF{I}_{J_2}, \ldots, \BF{I}_{J_N} \rrbracket$. 
		
	\item[(b)] 
		$(\ten{A}\modcon\ten{B})\modcon\ten{C} = \ten{A}\modcon(\ten{B}\modcon\ten{C})$. 
		
	\item[(c)] $\BF{P}\modcon\ten{G} = \ten{G}\times_1\BF{P}$.
	
	\item[(d)] $\ten{G}\modcon\BF{Q}=\ten{G}\times_N\BF{Q}^\Trps$.
	
	\item[(e)]
			 $\left( \ten{A}\modcon\ten{B} \right)_{\unfoldto{m}} =  
				\BF{A}_{\unfoldto{m}} \left( \BF{I}_{I_{m+1}I_{m+2}\cdots I_{M-1}} 
				\otimes  \BF{B}_{(1)} \right)$
			for $m=1,2,\ldots,M-1$.  
			
	\item[(f)] 
			 $\left( \ten{A}\modcon\ten{B} \right)_{\unfoldto{M+n-2}} = 
				\left( \BF{A}_{\unfoldto{M-1}} \otimes \BF{I}_{J_2J_3\cdots J_n} \right)
				\BF{B}_{\unfoldto{n}}$
			for $n=2,\ldots,N.$
	\item[(g)] 
			$\text{vec}\left( \ten{A}\modcon\ten{B} \right)
			= \left( \BF{I}_{I_{1}I_{2}\cdots I_{M-1}} 
				\otimes  \BF{B}_{(1)}^\Trps \right) \text{vec}\left(\ten{A}\right)
			= \left( \BF{A}_{(M)}^\Trps \otimes \BF{I}_{J_2J_3\cdots J_N} \right)
				\text{vec}\left( \ten{B} \right)$. 
	\end{enumerate}
\end{prop}
\begin{proof}
	(a) to (d) follow immediately from the definitions of the corresponding operations. 
	We can prove (e) and the first equality of (g) as follows.
	Let $\BF{Y}_{\unfoldto{0}} \equiv \text{vec}(\ten{Y})^\Trps$ denote the row vector
	for a tensor $\ten{Y}$.  
	 Note that $\ten{A}\modcon\ten{B}\in
	\BB{R}^{I_1\times \cdots\times I_{M-1}\times J_2\times \cdots \times J_N}$. 
	For $0\leq m\leq M-1$, we have
		\begin{equation*}
		\begin{split}
		\left(\ten{A}\modcon\ten{B}\right)_{\unfoldto{m}} 
		( \overline{i_1\cdots i_m}, \overline{i_{m+1}\cdots i_{M-1}j_2\cdots j_N}) 
		& = \left(\ten{A}\modcon\ten{B}\right)(i_1,\ldots,i_{M-1},j_2,\ldots,j_N)\\
		& = \sum_{i_M=1}^{I_M} \ten{A}(i_1,\ldots,i_M) \ten{B}(i_M,j_2,\ldots,j_N)\\
		& = \sum_{i_M=1}^{I_M} \ten{A}(i_1,\ldots,i_M)  \BF{B}_{(1)} (i_M, \overline{j_2\cdots j_N})\\
		& = \left( \ten{A} \times_M \BF{B}_{(1)}^\Trps \right) (i_1,\ldots,i_{M-1}, \overline{j_2\cdots j_N}). 
		\end{split}
		\end{equation*}
	From Proposition~\ref{prop_mode_n_mat}(f), we have 
		$$
		\ten{X} = \ten{A} \times_M \BF{B}_{(1)}^\Trps
		\Leftrightarrow 
		\BF{X}_{\unfoldto{m}} = \BF{A}_{\unfoldto{m}}
		\left( \BF{I}_{I_{m+1}} \otimes \cdots \otimes \BF{I}_{I_{M-1}} \otimes
		 \BF{B}_{(1)} \right), 
		$$
	and since
	$$\ten{X} (i_1,\ldots,i_{M-1},\overline{j_2\cdots j_N}) = 
	\BF{X}_{\unfoldto{m}} (\overline{i_1\cdots i_m}, \overline{i_{m+1}\cdots i_{M-1} j_2\cdots j_N}), $$
	the results in (e) and (g) follow. 
	
	We can prove (f)  and the second equality of (g) similarly. We have
		$$
		\left(\ten{A}\modcon\ten{B}\right)_{\unfoldto{M+n-2}} 
		( \overline{i_1\cdots i_{M-1}j_2\cdots j_{n}  }, \overline{j_{n+1}\cdots j_N}) 
		= \left( \ten{B} \times_1 \BF{A}_{\unfoldto{M-1}} \right) 
		( \overline{i_1\cdots i_{M-1}}, j_2,\ldots,j_N).  
		$$
	From Proposition~\ref{prop_mode_n_mat}(f), we have 
		$$
		\ten{X} = \ten{B} \times_1 \BF{A}_{\unfoldto{M-1}}
		\Leftrightarrow 
		\BF{X}_{\unfoldto{n}} = \left( \BF{A}_{\unfoldto{M-1}} \otimes \BF{I}_{J_2} \otimes 
				\BF{I}_{J_3} \otimes \cdots \otimes \BF{I}_{J_{n}}  \right)
				\BF{B}_{\unfoldto{n}}, 
		$$
	and the results in (f) and (g) follow from 
	$$\ten{X} ( \overline{i_1\cdots i_{M-1}}, j_2,\ldots,j_N) = 
	\BF{X}_{\unfoldto{n}} (\overline{i_1\cdots i_{M-1}j_2\cdots j_{n}  }, \overline{j_{n+1}\cdots j_N}). $$ 
		
\end{proof}

Moreover, the following property states that several binary operations 
(addition, direct sum, Hadamard product, and Kronecker product) preserve
the form of sequential contracted products of factor (core) tensors. The form will be 
introduced as the TT decomposition in Section~\ref{sec:ttd}. 

\begin{prop} 
\label{prop_TT_algebra}
Let $N\geq 2$ and 
 	\begin{equation*}
 	\begin{split}
	\ten{A}  
	&=  \BF{A}^{(1)} \modcon\ten{A}^{(2)} \modcon\cdots\modcon \ten{A}^{(N-1)}\modcon \BF{A}^{(N)}
	\in\BB{R}^{I_1\times I_2\times\cdots\times I_N},\\
	\qquad
	\ten{B}  
	&=  \BF{B}^{(1)}\modcon\ten{B}^{(2)} \modcon\cdots\modcon \ten{B}^{(N-1)}\modcon \BF{B}^{(N)}
	\in\BB{R}^{J_1\times J_2\times\cdots\times J_N}, 
	\end{split}
	\end{equation*}
where $\ten{A}^{(n)}$ and $\ten{B}^{(n)}$ are factor (core) tensors with sizes 
	$$
	\BF{A}^{(1)}\in\BB{R}^{I_1\times R_1}, \quad
	\ten{A}^{(n)}\in\BB{R}^{R_{n-1}\times I_n\times R_n},\ n=2,\ldots,N-1, \quad 
	\BF{A}^{(N)}\in\BB{R}^{R_{N-1}\times I_N}, 
	$$ 
	$$
	\BF{B}^{(1)}\in\BB{R}^{J_1\times S_1}, \quad
	\ten{B}^{(n)}\in\BB{R}^{S_{n-1}\times J_n\times S_n},\ n=2,\ldots,N-1, \quad 
	\BF{B}^{(N)}\in\BB{R}^{S_{N-1}\times J_N}. 
	$$
Then, 
	\begin{enumerate}
	\item[(a)] $\ten{A}+\ten{B} = 
			\left(\BF{A}^{(1)} \boxplus \BF{B}^{(1)} \right) \modcon 
			\left(\ten{A}^{(2)} \boxplus \ten{B}^{(2)} \right)
			\modcon \cdots \modcon
			\left(\BF{A}^{(N)} \boxplus \BF{B}^{(N)} \right)$
			if $\ten{A}$ and $\ten{B}$ have the same size. 
	\item[(b)] $\ten{A}\oplus\ten{B} = 
			\left(\BF{A}^{(1)}\oplus \BF{B}^{(1)} \right) \modcon
			\left(\ten{A}^{(2)}\oplus \ten{B}^{(2)} \right)
			\modcon \cdots \modcon
			\left(\BF{A}^{(N)}\oplus\BF{B}^{(N)} \right)$.
	\item[(c)] $\ten{A}\circledast\ten{B} = 
			\left(\BF{A}^{(1)} \boxtimes \BF{B}^{(1)} \right) \modcon
			\left(\ten{A}^{(2)} \boxtimes \ten{B}^{(2)} \right)
			\modcon \cdots \modcon
			\left(\BF{A}^{(N)} \boxtimes \BF{B}^{(N)} \right)$
			if $\ten{A}$ and $\ten{B}$ have the same size. 
	\item[(d)] $\ten{A}\otimes\ten{B} = 
			\left(\BF{A}^{(1)}\otimes \BF{B}^{(1)} \right) \modcon
			\left(\ten{A}^{(2)}\otimes \ten{B}^{(2)} \right)
			\modcon \cdots \modcon
			\left(\BF{A}^{(N)}\otimes\BF{B}^{(N)} \right)$.
	\end{enumerate}
\end{prop}
\begin{proof}
(a) to (d) can be derived by algebraic manipulation and the definitions of the corresponding 
operations. 
\end{proof}

The contracted product of tensors defined above can be further generalized to 
a contracted product of block tensors as follows. 
In the following definition, the tensors are given in
partitioned form and the contracted product is 
performed between each pair of blocks. 

\begin{defn}[Mode-$(M,1)$ contracted product for block tensors]
Let tensors $\widetilde{\ten{A}} = \begin{bmatrix}\ten{A}_{r_1,r_2} \end{bmatrix}$ 
and $\widetilde{\ten{B}} = \begin{bmatrix}\ten{B}_{s_1,s_2} \end{bmatrix}$
be block tensors partitioned with 
$M$th-order tensors $\ten{A}_{r_1,r_2}\in\BB{R}^{I_1\times \cdots \times I_M}$, 
$r_1=1,\ldots,R_1$, $r_2=1,\ldots,R_2$, and
$N$th-order tensors $\ten{B}_{s_1,s_2}\in\BB{R}^{J_1\times \cdots \times J_N}$ with $I_M = J_1$, $s_1=1,\ldots,S_1$, 
$s_2=1,\ldots,S_2$, respectively, i.e., 
	$$
	\widetilde{\ten{A}}
	= \begin{bmatrix}
	\ten{A}_{1,1}&\cdots&\ten{A}_{1,R_2}\\
	\vdots&\ddots&\vdots\\
	\ten{A}_{R_1,1}&\cdots&\ten{A}_{R_1,R_2}
	\end{bmatrix}, 
	\qquad 
	\widetilde{\ten{B}}
	= \begin{bmatrix}
	\ten{B}_{1,1}&\cdots&\ten{B}_{1,S_2}\\
	\vdots&\ddots&\vdots\\
	\ten{B}_{S_1,1}&\cdots&\ten{B}_{S_1,S_2}
	\end{bmatrix}. 
	$$ 
The mode-$(M,1)$ contracted product of $\widetilde{\ten{A}}$ and $\widetilde{\ten{B}}$ is 
defined by the block tensor
	$$
	\widetilde{\ten{C}}
	= \begin{bmatrix} \ten{C}_{t_1,t_2}
	\end{bmatrix}
	= 
	\widetilde{\ten{A}} \modcon \widetilde{\ten{B}}
	$$
partitioned with the $(M+N-1)$th-order tensors 
	$$
	\ten{C}_{t_1,t_2} = 
	\ten{A}_{r_1,r_2} \modcon \ten{B}_{s_1,s_2}
	\in\BB{R}^{I_1\times\cdots\times I_{M-1} \times J_2 \times \cdots \times J_N},
	\quad 
	t_1 = \overline{r_1s_1}, \ 
	t_2 = \overline{r_2s_2},
	$$
for all $t_1=1,\ldots,R_1S_1$, $t_2=1,\ldots,R_2S_2$. 

\end{defn}

\begin{exmp}
For example, the contracted product of block matrices $\widetilde{\BF{A}}$ and 
$\widetilde{\BF{B}}$ with 2 row partitions and 2 column partitions yields a block 
matrix with 4 row partitions and 4 column partitions as 
	$$
	\begin{bmatrix}\BF{A}_{11}&\BF{A}_{12}\\
	\BF{A}_{21}&\BF{A}_{22}\end{bmatrix}
	 \modcon
	\begin{bmatrix}\BF{B}_{11}&\BF{B}_{12}\\
	\BF{B}_{21}&\BF{B}_{22}\end{bmatrix}
	= 
	\begin{bmatrix}
	\BF{A}_{11} \BF{B}_{11} & \BF{A}_{11} \BF{B}_{12}&
	\BF{A}_{12} \BF{B}_{11}& \BF{A}_{12} \BF{B}_{12}\\
	\BF{A}_{11} \BF{B}_{21}& \BF{A}_{11} \BF{B}_{22}&
	\BF{A}_{12} \BF{B}_{21}&\BF{A}_{12} \BF{B}_{22}\\
	\BF{A}_{21} \BF{B}_{11} & \BF{A}_{21} \BF{B}_{12}&
	\BF{A}_{22} \BF{B}_{11}& \BF{A}_{22} \BF{B}_{12}\\
	\BF{A}_{21} \BF{B}_{21}& \BF{A}_{21} \BF{B}_{22}&
	\BF{A}_{22} \BF{B}_{21}&\BF{A}_{22} \BF{B}_{22}\\
	\end{bmatrix}.
	$$
Note that the above definition of contracted product for block matrices/block tensors
is different from the standard matrix-by-matrix product. That is, in the same example, the standard matrix-by-matrix product produces a block matrix with 2 row partitions and 2 column partitions  with blocks $\BF{C}_{t_1,t_2} = \sum_{r=1}^2 \BF{A}_{t_1,r} \BF{B}_{r,t_2}. $
\end{exmp}

\subsection{Strong Kronecker product}

The strong Kronecker product is an important tool for representation of 
low-rank TT decompositions of large-scale vectors, matrices, and low-order tensors. 
The original definition of the strong Kronecker product for block matrices \cite{Launey94} 
is presented below, together with its generalization to block tensors. 

\begin{defn}[Strong Kronecker product, \cite{Launey94}]
Let matrices $\BF{A} = [\BF{A}_{r_1,r_2}] \in\BB{R}^{R_1I_1\times R_2J_1}$ and 
$\BF{B} = [\BF{B}_{r_2,r_3}] \in\BB{R}^{R_2I_2 \times R_3J_2}$ 
be block matrices partitioned with $\BF{A}_{r_1,r_2}\in\BB{R}^{I_1\times J_1}$
and $\BF{B}_{r_2,r_3}\in\BB{R}^{I_2\times J_2}$, respectively. 
The strong Kronecker product of $\BF{A}$ and $\BF{B}$ is defined by 
the block matrix 
	$$
	\BF{C}
	= \begin{bmatrix} \BF{C}_{r_1,r_3} \end{bmatrix}
	= \BF{A} \skron \BF{B} \in\BB{R}^{R_1I_1I_2\times R_3J_1J_2}, 
	$$
partitioned with the $I_1I_2\times J_1J_2$ matrices 
	$$
	\BF{C}_{r_1,r_3} = \sum_{r_2=1}^{R_2} \BF{A}_{r_1,r_2}\otimes \BF{B}_{r_2,r_3}
	\in\BB{R}^{I_1I_2 \times J_1J_2}, 
	$$
for $	r_1=1,\ldots,R_1,$ $r_3=1,\ldots,R_3$.

More generally, let tensors 
$\ten{A} = \begin{bmatrix}\ten{A}_{r_1,r_2} \end{bmatrix} 
\in\BB{R}^{R_1I_1\times R_2J_1 \times K_1}$ and 
$\ten{B} = \begin{bmatrix}\ten{B}_{r_2,r_3} \end{bmatrix}
\in\BB{R}^{R_2I_2\times R_3J_2 \times K_2}$ 
be block tensors partitioned with 3rd-order tensors 
$\ten{A}_{r_1,r_2}\in\BB{R}^{I_1\times J_1\times K_1}$ and 
$\ten{B}_{r_2,r_3}\in\BB{R}^{I_2\times J_2\times K_2}$. 
The strong Kronecker product of $\ten{A}$ and $\ten{B}$ is 
defined by the block tensor
	$$
	\ten{C}
	= \begin{bmatrix} \ten{C}_{r_1,r_3}
	\end{bmatrix}
	= 
	\ten{A} \skron \ten{B}
	\in\BB{R}^{R_1I_1I_2 \times R_3J_1J_2 \times K_1K_2}, 
	$$
partitioned with the $I_1I_2 \times J_1J_2 \times K_1K_2$ tensors 
	$$
	\ten{C}_{r_1,r_3} = \sum_{r_2=1}^{R_2} 
	\ten{A}_{r_1,r_2}\otimes \ten{B}_{r_2,r_3}
	\in\BB{R}^{I_1I_2 \times J_1J_2 \times K_1K_2},
	$$
for $	r_1=1,\ldots,R_1$, $r_3=1,\ldots,R_3$.
\end{defn}

\begin{exmp}
The strong Kronecker product has a similarity with the matrix-by-matrix multiplication. 
For example, 
	$$
	\begin{bmatrix}\BF{A}_{11}&\BF{A}_{12}\\
	\BF{A}_{21}&\BF{A}_{22}\end{bmatrix}
	 \skron 
	\begin{bmatrix}\BF{B}_{11}&\BF{B}_{12}\\
	\BF{B}_{21}&\BF{B}_{22}\end{bmatrix}
	= 
	\begin{bmatrix}\BF{A}_{11}\otimes\BF{B}_{11}
	+\BF{A}_{12}\otimes\BF{B}_{21}&
	\BF{A}_{11}\otimes\BF{B}_{12}+\BF{A}_{12}\otimes\BF{B}_{22}\\
	\BF{A}_{21}\otimes\BF{B}_{11}+\BF{A}_{22}\otimes\BF{B}_{21}&
	\BF{A}_{21}\otimes\BF{B}_{12}+\BF{A}_{22}\otimes\BF{B}_{22}
	\end{bmatrix}.
	$$
\end{exmp}

%

\subsection{Partial trace operator}

We will define a linear operator, $\ptrace$, called as the partial trace, which generalizes 
the trace on matrices to tensors. 

\begin{defn}[Partial trace operator]
The partial trace, $\ptrace:\BB{R}^{R\times I_1\times I_2\times \cdots \times I_{N} \times R}
\rightarrow\BB{R}^{I_1\times I_2\times \cdots \times I_N}$, 
$N\geq 2$, is a linear operator defined by the tensor $\ten{Y}=\ptrace(\ten{X})$ with entries 
	\begin{equation}\label{def:tentr}
	\ten{Y}(i_1,i_2,\ldots,i_N) = 
	\sum_{r=1}^{R} \ten{X}\left(r,i_1,i_2,\ldots,i_N,r \right). 
	\end{equation}
\end{defn}

A more formal definition can be given by using the contracted product as 
		\begin{equation}\label{def:tentrByProd}
		\ptrace: \ten{X} \mapsto 
		\ptrace(\ten{X}) =
		\sum_{r=1}^{R} \BF{e}_{r}\modcon \ten{X} \modcon \BF{e}_{r}, 
		\end{equation}
where $\BF{e}_{r}=\left[ 0, \ldots, 0, 1, 0,\ldots, 0\right]^\Trps \in\BB{R}^{R}$
is the $r$th standard basis vector. 
The partial trace is a generalization of the matrix trace: for a matrix 
$\BF{A}\in\BB{R}^{R\times R}$, $\ptrace( \BF{A} ) = \text{trace}(\BF{A}) = \sum_{r=1}^R a_{rr} \in\BB{R}.$
For a tensor $\ten{X}\in\BB{R}^{R\times I_1\times I_2\times \cdots \times I_N\times R}$, 
the $(i_1,i_2,\ldots,i_{N})$th entry of $\ptrace(\ten{X})$ equals
to the matrix trace of the $(i_1,i_2,\ldots,i_{N})$th slice as
	\begin{equation}
	\left( \ptrace(\ten{X}) \right)_{i_1,i_2,\ldots,i_N}
	 = \text{trace}\left( \BF{X}_{:,i_1,i_2,\ldots,i_N,:}  \right). 
	\end{equation}
Figure~\ref{Fig_partial_trace}(a) illustrates a tensor network diagram representing 
the partial trace of a 7th-order tensor. 

\begin{figure}
\centering
\begin{tabular}{ccc}
\includegraphics[width=3cm]{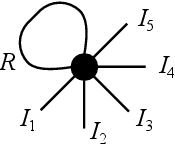}& \hspace{2pc} & 
\includegraphics[width=4.2cm]{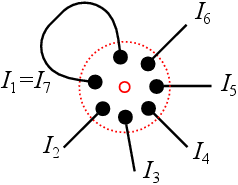}
\\
(a) & & (b)
\end{tabular}
\caption{\label{Fig_partial_trace}
Tensor network diagrams for (a) partial trace, 
$\ptrace ( \ten{X} )$, of a tensor $\ten{X}\in\BB{R}^{R\times I_1\times\cdots\times I_5\times R}$, and 
(b) partial trace of a rank-one tensor of order $N=7$
as described in Example~\ref{example_partial_trace_rank1}. 
}
\end{figure}

\begin{exmp}
\label{example_partial_trace_rank1}
The partial trace of a rank-one tensor can be expressed by, for $N\geq 1$, 
	$$
	\ptrace\left( \BF{a}^{(1)}\circ\BF{a}^{(2)}\circ\cdots
	\circ\BF{a}^{(N)} \right)
	= \langle \BF{a}^{(1)}, \BF{a}^{(N)} \rangle
	\cdot  \BF{a}^{(2)}\circ\cdots
	\circ\BF{a}^{(N-1)} , 
	$$
where $\langle \BF{v}, \BF{w} \rangle = \BF{v}^\Trps\BF{w}$, 
see, e.g., Figure~\ref{Fig_partial_trace}(b). 
\end{exmp}

\begin{exmp}
If $\ten{A}\in\BB{R}^{R_1\times I\times R_2}$, 
$\ten{B}\in\BB{R}^{R_2\times J\times R_3}$, and 
$\ten{C}\in\BB{R}^{R_3\times K\times R_1}$, then, 
	$$
	\left( \ptrace(\ten{A}\modcon\ten{B} \modcon\ten{C} ) \right)_{ijk} = 
	\text{trace}(\BF{A}_{:,i,:}\BF{B}_{:,j,:}\BF{C}_{:,k,:}) = 
	\left( \ptrace(\ten{C}\modcon \ten{B}\modcon\ten{A}) \right)_{kji}. 
	$$
\end{exmp}

\section{Tensor decompositions}

Tensor decomposition is an approximate representation of a tensor as contraction of a set of 
factor (core) tensors. See, e.g., \cite{EspHacHanSch2011,Hac2012} for more general definitions of 
tensor format and tensor representation in tensor product space. 
We will introduce several tensor decompositions (CP, Tucker, TT) in a consistent way, 
based on the notations and tensor operations defined in the previous section.

   


\subsection{CP and Tucker decompositions} 
\label{sec:cp_n_tucker}


%

The CANDECOMP/PARAFAC (CP) decomposition expresses a tensor 
as a sum of rank-one tensors:  a tensor $\ten{X}$ of size 
$I_1\times I_2\times \cdots\times I_N$ is written as 
	\begin{equation} \label{eqn:CPDelem}
	\ten{X} = 
	\sum_{r=1}^R \lambda_r \BF{a}^{(1)}_r 
	\circ \BF{a}^{(2)}_r \circ\cdots \circ \BF{a}^{(N)}_r, 
	\end{equation}
where $\BF{a}^{(n)}_r\in\BB{R}^{I_n}, r=1,\ldots,R,n=1,\ldots,N,$ are normalized 
vectors, $\lambda_r,r=1,\ldots,R,$ are weights, and $R\in\BB{N}$ is called the rank of the CP 
decomposition \eqref{eqn:CPDelem}. The above expression can be equivalently re-written 
in form of multilinear products as 
	\begin{equation} \label{eqn:CPD}
	\ten{X}
	 = 
	\ten{\Lambda} \times_1 \BF{A}^{(1)} \times_2 \BF{A}^{(2)} \ldots \times_N \BF{A}^{(N)}
	 = 
	\left\llbracket \ten{\Lambda} ; \BF{A}^{(1)}, \BF{A}^{(2)}, \ldots, \BF{A}^{(N)} \right\rrbracket, 
	\end{equation}
where $\ten{\Lambda}\in\BB{R}^{R\times \cdots \times R}$ is the superdiagonal tensor 
with diagonals $\lambda_1,\ldots,\lambda_R$, and 
$\BF{A}^{(n)}=[\BF{a}^{(n)}_1,\ldots,\BF{a}^{(n)}_R]$ are $I_n\times R_n$
factor matrices. 
Other alternative representations are summarized in Table \ref{Table:notation_CP_n_Tucker}. 
The Khatri-Rao (columnwise Kronecker) product of two matrices 
$\BF{A}\in\BB{R}^{I\times R}$ and $\BF{B}\in\BB{R}^{J\times R}$ is denoted by 
$\BF{C}= \BF{A}\odot\BF{B}\in\BB{R}^{IJ\times R}$ with columns 
$\BF{c}_r = \BF{a}_r\otimes \BF{b}_r$, $r=1,\ldots,R$.

\begin{table} 
\centering
\caption{\label{Table:notation_CP_n_Tucker}
Various representations for the CP and Tucker decompositions
of a tensor $\ten{X}$ of size $I_1\times I_2\times\cdots\times I_N$. }
\vspace{.5pc}
\begin{tabular}{ll}
\hline\\[-0.8pc]
\multicolumn{1}{c}{CP}   &    \multicolumn{1}{c}{Tucker}\\
\hline
\multicolumn{2}{c}{Multilinear product} \\
	$\displaystyle
	\ten{X} = \ten{\Lambda} \times_1 \BF{A}^{(1)} \cdots \times_N \BF{A}^{(N)} 
	$
	& 
	$\displaystyle
	\ten{X} = \ten{G} \times_1 \BF{A}^{(1)} \cdots \times_N \BF{A}^{(N)}
	$
\\[0.2pc]
	$\displaystyle
	\ten{X} = \left\llbracket \ten{\Lambda}; \BF{A}^{(1)}, \ldots, \BF{A}^{(N)} \right\rrbracket 
	$
	& 
	$\displaystyle
	\ten{X} = \left\llbracket \ten{G}; \BF{A}^{(1)},\ldots,\BF{A}^{(N)} \right\rrbracket
	$
\\[0.4pc]
\hline
\multicolumn{2}{c}{Outer product} \\
	$\displaystyle  
	\ten{X} = \sum_{r=1}^R \lambda_r \BF{a}^{(1)}_r \circ \cdots \circ \BF{a}^{(N)}_r
	$
	& 
	$\displaystyle  
	\ten{X} = \sum_{r_1=1}^{R_1} \cdots\sum_{r_N=1}^{R_N} g_{r_1,\ldots,r_N} \BF{a}^{(1)}_{r_1} \circ\cdots \circ \BF{a}^{(N)}_{r_N}
	$
\\
\hline 
\multicolumn{2}{c}{Scalar product} \\
	$\displaystyle  
	x_{i_1,\ldots,i_N} = \sum_{r=1}^R \lambda_r a^{(1)}_{i_1,r} \cdots a^{(N)}_{i_N,r}
	$
	& 
	$\displaystyle  
	x_{i_1,\ldots,i_N} = \sum_{r_1=1}^{R_1} \cdots\sum_{r_N=1}^{R_N} g_{r_1,\ldots,r_N} a^{(1)}_{i_1,r_1} \cdots a^{(N)}_{i_N,r_N}
	$
\\
\hline 
\multicolumn{2}{c}{Slice representation*} \\
	$\displaystyle 
	\BF{X}_{:,:,i_3,\ldots,i_N} = \BF{A}^{(1)} \widetilde{\BF{D}}_{i_3,\ldots,i_N} \BF{A}^{(2)\text{T}}
	$
	& 
	$\displaystyle 
	\BF{X}_{:,:,i_3,\ldots,i_N} = \BF{A}^{(1)}\widetilde{\BF{G}}_{i_3,\ldots,i_N} \BF{A}^{(2)\text{T}}
	$
\\
\hline 
\multicolumn{2}{c}{Vectorization**} \\
	$\displaystyle 
	\text{vec}(\ten{X}) = \left( \bigodot_{n=1}^N \BF{A}^{(n)} \right) \boldsymbol\lambda
	$
	& 
	$\displaystyle 
	\text{vec}(\ten{X}) = \left( \bigotimes_{n=1}^N \BF{A}^{(n)} \right) \text{vec}(\ten{G})
	$
\\
\hline 
\multicolumn{2}{c}{Matricization**} \\
	$\displaystyle 
	\BF{X}_{(n)} = \BF{A}^{(n)} \boldsymbol\Lambda 
	\left( \bigodot_{\myatop{m=1}{m\neq n}}^N \BF{A}^{(m)} \right)^\text{T}
	$
	& 
	$\displaystyle 
	\BF{X}_{(n)} = \BF{A}^{(n)} \BF{G}_{(n)}
	\left( \bigotimes_{\myatop{m=1}{m\neq n}}^N \BF{A}^{(m)} \right)^\text{T}
	$
\\
	$\displaystyle 
	\BF{X}_{\unfoldto{n}} = \left( \bigodot_{m=1}^n \BF{A}^{(m)} \right) \boldsymbol\Lambda 
	\left( \bigodot_{m=n+1}^N \BF{A}^{(m)} \right)^\text{T}
	$
	&
	$\displaystyle 
	\BF{X}_{\unfoldto{n}} = \left( \bigotimes_{m=1}^n \BF{A}^{(m)} \right) \BF{G}_{\unfoldto{n}}
	\left( \bigotimes_{m=n+1}^N \BF{A}^{(m)} \right)^\text{T}
	$
\\
\hline\\[-0.7pc]
\multicolumn{2}{l}{
	* $\widetilde{\BF{D}}_{i_3,\ldots,i_N} = \text{diag}(\tilde{d}_{11},\ldots,\tilde{d}_{RR}) \in\BB{R}^{R\times R}$ with diagonals
	$\tilde{d}_{rr}=\lambda_ra^{(3)}(i_3,r)\cdots a^{(N)}(i_N,r)$.}
\\
\multicolumn{2}{l}{
	* $\widetilde{\BF{G}}_{i_3,\ldots,i_N} =
	\sum_{r_3}\cdots\sum_{r_N} a^{(3)}_{i_3,r_3}\cdots a^{(N)}_{i_N,r_N} \BF{G}_{:,:,r_3,\ldots,r_N}$
	is the sum of frontal slices.}
\\
\multicolumn{2}{l}{
	** $\BF{A}\odot\BF{B}$ stands for the Khatri-Rao (columnwise Kronecker) product of matrices
	$\BF{A}$ and $\BF{B}$.}
\end{tabular}
\end{table}

The Tucker decomposition \cite{tucker66}   decomposes a tensor into 
a core tensor multiplied by a factor matrix on each mode as
	\begin{equation}\label{eqn:Tucker}
	\ten{X} 
	= \ten{G}\times_1\BF{A}^{(1)} \times_2\BF{A}^{(2)}
	\cdots\times_N\BF{A}^{(N)}
	= \left\llbracket 
	\ten{G} ; \BF{A}^{(1)} , \BF{A}^{(2)}, \ldots, 
	\BF{A}^{(N)} \right\rrbracket, 
	\end{equation}
where $\ten{G}\in\BB{R}^{R_1\times R_2\times\cdots\times R_N}$
is a core tensor, $\BF{A}^{(n)}\in\BB{R}^{I_n\times R_n}$
are factor matrices, and $(R_1,\ldots,R_N)$ is called the (multilinear) rank 
of the Tucker decomposition \eqref{eqn:Tucker}. 
The Tucker decomposition 
can also be represented as a sum of rank-one tensors 
as 
	\begin{equation}
	\ten{X} = 
	\sum_{r_1=1}^{R_1} \sum_{r_2=1}^{R_2} \cdots\sum_{r_N=1}^{R_N} 
	g_{r_1,r_2,\ldots,r_N} \BF{a}^{(1)}_{r_1} \circ \BF{a}^{(2)}_{r_2} \circ
	\cdots \circ \BF{a}^{(N)}_{r_N}, 
	\end{equation}
where $\BF{a}^{(n)}_{r_n} \in\BB{R}^{I_n}$ is the $r_n$th column of $\BF{A}^{(n)}$. 
Alternative representations for the Tucker decomposition are 
summarized in Table \ref{Table:notation_CP_n_Tucker}. 

In general, the CP can be regarded as a special case of the Tucker in the sense that 
the CP is a Tucker decomposition with a superdiagonal core tensor, e.g., 
see \eqref{eqn:CPD} and \eqref{eqn:Tucker}, where the multilinear rank is $(R,R,\ldots,R)$. 


The CP and Tucker decompositions can be illustrated by tensor network diagrams 
as in Figure~\ref{Fig:CP_n_Tucker}. Although the CP decomposition was illustrated 
by a network diagram in the figure, it is often not classified as a 
{\em tensor network format} 
when $N\geq 3$ and $R\geq 2$ in a strict sense \cite{EspHacHanSch2011},  
due to the superdiagonality of the core tensor $\ten{\Lambda}$.

\begin{figure}
\centering
\begin{tabular}{ccc}
\includegraphics[width=4cm]{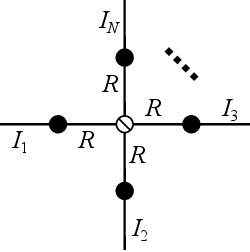}  & \hspace{2pc} & 
\includegraphics[width=4cm]{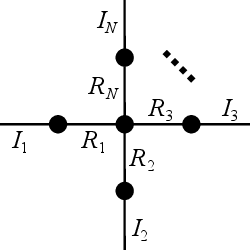}  \\
(a) CP decomposition & & (b) Tucker decomposition
\end{tabular}
\caption{\label{Fig:CP_n_Tucker}
Tensor network diagrams for the (a) CP and (b) Tucker decompositions of 
an $N$th-order tensor $\ten{X}\in\BB{R}^{I_1\times \cdots \times I_N}$. 
It is shown that the CP has a superdiagonal core tensor (denoted by 
a white circle with slash).}
\end{figure}

On the other hand, the higher-order singular value decomposition (HOSVD) \cite{DelathMoorVan2000} of a tensor produces a Tucker decomposition 
with orthogonalized factor matrices and a core tensor. 
All-orthogonality of the core tensor of Tucker decomposition is 
defined as follows: 
\begin{defn}[All-orthogonality, \cite{DelathMoorVan2000}]
An $N$th-order tensor $\ten{G}\in\BB{R}^{R_1\times \cdots \times R_N}$ is called 
all-orthogonal if 
	$$
	\BF{G}_{(n)}\BF{G}_{(n)}^\Trps = \BF{\Lambda}^{(n)}
	$$ 
for some diagonal matrices $\BF{\Lambda}^{(n)}=\text{diag}(\lambda^{(n)}_1,\ldots,\lambda^{(n)}_{R_n})\in\BB{R}^{R_n\times R_n}$ 
with diagonals $\lambda^{(n)}_1\geq \cdots \geq \lambda^{(n)}_{R_n}\geq 0$
for all $n=1,\ldots,N$. 
\end{defn}

\subsection{Tensor train (TT) decomposition}
\label{sec:ttd}

By the tensor train (TT) decomposition, a tensor 
$\ten{X}\in\BB{R}^{I_1\times I_2\times\cdots\times I_N}$ is represented as
contracted products
	\begin{equation}\label{eqn:TTcontract}
	\ten{X} = 
	\ten{G}^{(1)}\modcon \ten{G}^{(2)}
	\modcon\cdots \modcon \ten{G}^{(N)}, 
	\end{equation}
where $\ten{G}^{(n)}$ are 3rd-order core (factor) tensors called the TT-cores
with sizes $R_{n-1}\times I_n\times R_n$, $n=1,\ldots,N$, the integers 
$R_1,\ldots,R_{N-1}$ are called the TT-ranks, and  we assume that $R_0=R_N=1$. 
For notational convenience, we consider the 1st and the $N$th TT-cores as 3rd-order tensors unless stated otherwise.  
Figure~\ref{Fig:TTnotNormalized} illustrates the tensor network diagram for the
TT decomposition of a $7$th-order tensor. 

The TT decomposition can alternatively be written entrywise as products of 
slice matrices 
	\begin{equation}\label{eqn:TTentrywise}
	x_{i_1,i_2,\ldots,i_N} 
	= 
	\BF{G}^{(1)}_{i_1} \BF{G}^{(2)}_{i_2}
	\cdots \BF{G}^{(N)}_{i_N},
	\end{equation}
where $\BF{G}^{(n)}_{i_n}=\ten{G}^{(n)}(:,i_n,:)\in\BB{R}^{R_{n-1}\times R_n}$
is the lateral slice of the $n$th TT-core, $n=1,\ldots,N$. In this case, 
we suppose that $\BF{G}^{(1)}_{i_1} \in\BB{R}^{1\times R_1}$ and 
$\BF{G}^{(N)}_{i_N} \in\BB{R}^{R_{N-1}\times 1}$ are row and column vectors. 
The TT decomposition can also be represented by a sum of outer products
	\begin{equation}\label{eqn:TTouter}
	\ten{X} = 
	\sum_{r_1=1}^{R_1}\sum_{r_2=1}^{R_2}\cdots
	\sum_{r_{N-1}=1}^{R_{N-1}}
	\BF{g}^{(1)}_{1,r_1} \circ
	\BF{g}^{(2)}_{r_1,r_2} \circ
	\cdots \circ
	\BF{g}^{(N-1)}_{r_{N-2},r_{N-1}} \circ
	\BF{g}^{(N)}_{r_{N-1},1},
	\end{equation}
where $\BF{g}^{(n)}_{r_{n-1},r_n}=\ten{G}^{(n)}(r_{n-1},:,r_n)\in\BB{R}^{I_n}$
are the mode-2 fibers. 

\begin{figure}
\centering
\includegraphics[width = 7.3cm]{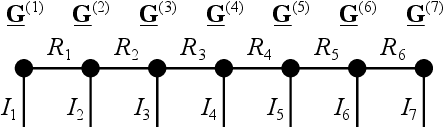}
\caption{\label{Fig:TTnotNormalized}
Tensor network diagram for TT decomposition of a 7th-order tensor $\ten{X}\in\BB{R}^{I_1\times I_2\times \cdots \times I_7}$. }
\end{figure}

Table~\ref{Table:notation_TT_global_n_recursive} summarizes 
the various representations for the TT decomposition.

\begin{table} 
\centering
\caption{\label{Table:notation_TT_global_n_recursive}
Various representations for the TT decomposition 
of a tensor $\ten{X}$ of size $I_1\times\cdots\times I_N$, 
either in global representation (left column) or in recursive representation (right column).
$\ten{X} = \ten{G}^{\leq N} = \ten{G}^{\geq 1}.$}
\vspace{.5pc}
\begin{tabular}{ll}
\hline\\[-0.8pc]
\multicolumn{1}{c}{TT (global)}   &   \multicolumn{1}{c}{TT (recursive)}\\
\hline
\multicolumn{2}{c}{Contracted product} \\
	$\displaystyle
	\ten{X} = 	\ten{G}^{(1)}\modcon \ten{G}^{(2)}  \modcon\cdots \modcon \ten{G}^{(N)}
	$
	& 
	$\displaystyle
	\ten{G}^{\leq n} = \ten{G}^{\leq n-1} \modcon \ten{G}^{(n)}
	$ 
\\[0.2pc]
	&
	$\displaystyle
	\ten{G}^{\geq n} =  \ten{G}^{(n)} \modcon \ten{G}^{\geq n+1}
	$ 
\\
\hline 
\multicolumn{2}{c}{Outer product} \\
	$\displaystyle 
	\ten{X} = 	
		\sum_{r_1,\ldots,r_{N-1}=1}^{R_1,\ldots,R_{N-1}} 
		\BF{g}_{1,r_1}^{(1)}\circ \BF{g}_{r_1,r_2}^{(2)}
		\circ \cdots \circ \BF{g}_{r_{N-1},1}^{(N)}
	$
	&
	$\displaystyle 
		\ten{G}^{\leq n}_{r_n} = 
		\sum_{r_{n-1}=1}^{R_{n-1}} \ten{G}^{\leq n-1}_{r_{n-1}}
		\circ \BF{g}_{r_{n-1},r_n}^{(n)}
	$
\\
	& 
	$\displaystyle 
		\ten{G}^{\geq n}_{r_{n-1}} = 
		\sum_{r_n=1}^{R_n} \BF{g}_{r_{n-1},r_{n}}^{(n)}\circ
		\ten{G}^{\geq n+1}_{r_{n}}
	$
\\
\hline
\multicolumn{2}{c}{Scalar product} \\
	$\displaystyle
	x_{i_1,i_2,\ldots,i_N} = 
	$
	&
	$\displaystyle
	g^{\leq n}_{1,i_1,\ldots,i_n,r_n} = 
	$
\\
		\quad 
		$\displaystyle
		\sum_{r_1,\ldots,r_{N-1}=1}^{R_1,\ldots,R_{N-1}}
		g^{(1)}_{1,i_1,r_1} g^{(2)}_{r_1,i_2,r_2}
		\cdots g^{(N)}_{r_{N-1},i_N,1}
		$
	&	
		\quad
		$\displaystyle
		\sum_{r_{n-1}=1}^{R_{n-1}} 
		g^{\leq n-1}_{1,i_1,\ldots,i_{n-1},r_{n-1}}
		g^{(n)}_{r_{n-1},i_n,r_n}
		$
\\
	& 
	$\displaystyle
	g^{\geq n}_{r_{n-1}, i_n,\ldots,i_N,1} = 
	$
\\
	&
		\quad
		$\displaystyle
		\sum_{r_n=1}^{R_n} 
		g^{(n)}_{r_{n-1},i_n,r_n}
		g^{\geq n+1}_{r_n,i_{n+1},\ldots,i_N,1}
		$
\\
\hline 
\multicolumn{2}{c}{Matrix product} \\
	$\displaystyle 
	x_{i_1,i_2,\ldots,i_N} = 
	\BF{G}^{(1)}_{i_1} 
	\BF{G}^{(2)}_{i_2}
	\cdots 	\BF{G}^{(N)}_{i_N}
	$
	&
	$\displaystyle 
	\BF{G}^{\leq n}_{i_1,\ldots,i_n} = 
	\BF{G}^{\leq n-1}_{i_1,\ldots,i_{n-1}} 
	\BF{G}^{(n)}_{i_n}
	$
\\[0.2pc]
	&
	$\displaystyle 
	\BF{G}^{\geq n}_{i_n,\ldots,i_N} = 
	\BF{G}^{(n)}_{i_n} 
	\BF{G}^{\geq n+1}_{i_{n+1},\ldots,i_N}
	$
\\[0.2pc]
\hline
\multicolumn{2}{c}{Vectorization} \\
	$\displaystyle 
	\text{vec}(\ten{X}) 
	= \prod_{n=N}^1
	\left( \BF{I}_{I_1I_2\cdots I_{n-1}} \otimes \BF{G}^{(n)\Trps}_{(1)} \right)
	$
	&
	$\displaystyle
	\text{vec} (\ten{G}^{\leq n} ) =
	$
\\[0.4pc]
	& 
		\quad
		$\displaystyle
		\left( \BF{I}_{I_1I_2\cdots I_{n-1}} \otimes 
			\BF{G}^{(n)\Trps}_{(1)} \right)
		\text{vec} \left(\ten{G}^{\leq n-1}\right)
		$
\\[0.4pc]
	$\displaystyle 
	\text{vec}(\ten{X}) 
	= \prod_{n=1}^N \left( \BF{G}^{(n)\Trps}_{(3)} \otimes \BF{I}_{I_{n+1}\cdots I_N} \right)
	$
	&
	$\displaystyle
	\text{vec}\left(\ten{G}^{\geq n}\right) = 
	$
\\[0.4pc]
	&	
		\quad
		$\displaystyle
		\left(  \BF{G}^{(n)\Trps}_{(3)} \otimes 
			\BF{I}_{I_{n+1}I_{n+2}\cdots I_N} \right)
		\text{vec}\left(\ten{G}^{\geq n+1}\right)
		$
\\[0.4pc]
	{
	$\displaystyle 
	\text{vec}(\ten{X}) = 
	\left( (\BF{G}^{<n}_{(n)} )^\Trps \otimes \BF{I}_{I_n}\otimes
		(\BF{G}^{>n}_{(1)} )^\Trps \right) 
		\text{vec} \left(\ten{G}^{(n)} \right)	
	$
	}
	&
	
\\[0.4pc]
\hline
\multicolumn{2}{c}{Matricization} \\
	$\displaystyle
	\BF{X}_{(n)} = 
	\BF{G}^{(n)}_{(2)} 
	\left( \BF{G}^{< n}_{(n)} \otimes \BF{G}^{>n}_{(1)} \right)
	$
	& 
	$\displaystyle
	\BF{G}^{<n}_{(n)} = \BF{G}^{(n-1)}_{(3)}
	\left( \BF{G}^{<n-1}_{(n-1)} \otimes \BF{I}_{I_{n-1}} \right)
	$
\\[0.4pc]
	$\displaystyle
	\BF{X}_{\unfoldto{n}} = \BF{G}^{\leq n}_{\unfoldto{n}} \BF{G}^{>n}_{(1)}
	= \left( \BF{G}^{< n+1}_{(n+1)} \right)^\Trps \BF{G}^{>n}_{(1)}
	$
	&
	$\displaystyle
	\BF{G}^{>n}_{(1)} = \BF{G}^{(n+1)}_{(1)}
	\left( \BF{I}_{I_{n+1}} \otimes \BF{G}^{>n+1}_{(1)} \right)
	$
\\[0.2pc]
\hline
\end{tabular}
\end{table}

\subsubsection{Recursive representations and vectorizations for TT decomposition}

The TT decomposition can be expressed in a recursive manner, which is summarized
in Table~\ref{Table:notation_TT_global_n_recursive}. Given the TT-cores $\ten{G}^{(n)},n=1,\ldots,N,$ 
from a TT decomposition \eqref{eqn:TTcontract}, we define the partial contracted 
products $\ten{G}^{\leq n}$ and $\ten{G}^{\geq n}$ as 
	\begin{equation}
	\begin{split}
	\ten{G}^{\leq n}
	& = 
	\ten{G}^{(1)}\modcon\ten{G}^{(2)}\modcon
	\cdots\modcon\ten{G}^{(n)}
	\in\BB{R}^{I_1\times\cdots\times I_n\times R_n}, 
	\\
	\ten{G}^{\geq n}
	& = 
	\ten{G}^{(n)}\modcon\ten{G}^{(n+1)}\modcon
	\ldots\modcon\ten{G}^{(N)} \in\BB{R}^{R_{n}\times I_{n+1}\times\cdots\times I_N}, 
	\end{split}
	\end{equation}
for $n=1,\ldots,N$. $\ten{G}^{< n}$ and $\ten{G}^{> n}$ are defined in the same way. 
For completeness, we define $\ten{G}^{<1} = \ten{G}^{>N} = 1$. 
Vectorization of the partial contracted products yield the following recursive equations
(see also Proposition~\ref{prop_modcon}(g)): 
	\begin{equation} \label{eqn_recursive_vec}
	\begin{split}
	\text{vec}\left(\ten{G}^{\leq n}\right)
	& = 
	\text{vec}\left(\ten{G}^{\leq n-1} \times^1 \ten{G}^{(n)} \right) \\
	& = 
	\left( \BF{I}_{I_1I_2\cdots I_{n-1}}  \otimes \BF{G}^{(n)\Trps}_{(1)} \right)
	\text{vec}\left(\ten{G}^{\leq n-1}\right), 
	\quad 
	n=2,3,\ldots,N, 
	\\
	\text{vec}\left(\ten{G}^{\geq n}\right)
	& = 
	\text{vec}\left(  \ten{G}^{(n)} \times^1 \ten{G}^{\geq n+1}  \right) \\
	& =  \left(\BF{G}^{(n)\Trps}_{(3)} \otimes \BF{I}_{I_{n+1}I_{n+2}\cdots I_N}\right)
	\text{vec}\left(\ten{G}^{\geq n+1}\right), 
	\quad 
	n=1,2,\ldots,N-1. 
	\end{split}
	\end{equation}
Since $\text{vec} ( \ten{X} ) = \text{vec}( \ten{G}^{\leq N} )
= \text{vec}( \ten{G}^{\geq 1} )$, we can obtain 
the following formulas. 

\begin{prop}
\label{prop_recursive_equations}
Let $\ten{X}\in\BB{R}^{I_1\times I_2\times \cdots \times I_N}$ admit the 
TT decomposition in \eqref{eqn:TTcontract}. Then
	\begin{enumerate}
	\item[(a)]
		$\displaystyle
		\text{vec}(\ten{X}) 
		= \prod_{n=N}^1 \left( \BF{I}_{I_1I_2\cdots I_{n-1}} \otimes \BF{G}^{(n)\Trps}_{(1)} \right), 
		$
		with 
		$
		\BF{I}_{I_1I_0} = 1. 
		$
	
	\item[(b)]
		$\displaystyle
		\text{vec}(\ten{X}) 
		= \prod_{n=1}^N
		\left( \BF{G}^{(n)\Trps}_{(3)} \otimes \BF{I}_{I_{n+1}I_{n+2}\cdots I_N} \right), 
		$
		with 
		$
		\BF{I}_{I_{N+1}I_N} = 1. 
		$
		
	\item[(c)]
		$\displaystyle
		\text{vec}\left( \ten{X} \right)
		= \left( \left(\BF{G}^{<n}_{(n)}\right)^\Trps
				\otimes \BF{I}_{I_n} \otimes
				\left(\BF{G}^{> n}_{(1)}\right)^\Trps 
		\right)
		\text{vec}\left( \ten{G}^{(n)} \right). 
		$
		
	\end{enumerate}
\end{prop}
\begin{proof}
(a) and (b) follow immediately from \eqref{eqn_recursive_vec} and
that $\BF{G}^{(1)\Trps}_{(1)} = \text{vec}(\ten{G}^{(1)})$, 
$\BF{G}^{(N)\Trps}_{(3)} = \text{vec}(\ten{G}^{(N)})$. 
(c) can be derived by using Proposition~\ref{prop_modcon}(g) as 
	\begin{equation*}
	\begin{split}
	\text{vec}\left( \ten{X} \right) 
	& = \text{vec}\left( \ten{G}^{<n} \modcon \left(\ten{G}^{(n)} \modcon \ten{G}^{>n} \right)\right)\\
	& = \left( \left(\BF{G}^{<n}_{(n)}\right)^\Trps \otimes \BF{I}_{I_nI_{n+1}\cdots I_N}  \right)
	\text{vec}\left( \ten{G}^{(n)} \modcon \ten{G}^{>n} \right)\\
	& = \left( \left(\BF{G}^{<n}_{(n)}\right)^\Trps \otimes \BF{I}_{I_nI_{n+1}\cdots I_N}  \right)
	\left( \BF{I}_{R_{n-1}I_n} \otimes \left(\BF{G}^{>n}_{(1)}\right)^\Trps \right)
	\text{vec}\left( \ten{G}^{(n)} \right)\\
	& = \left( \left(\BF{G}^{<n}_{(n)}\right)^\Trps
				\otimes \BF{I}_{I_n} \otimes
				\left(\BF{G}^{> n}_{(1)}\right)^\Trps 
		\right)
		\text{vec}\left( \ten{G}^{(n)} \right). 
	\end{split}
	\end{equation*}

\end{proof}

The expression in Proposition~\ref{prop_recursive_equations}(c)
shows that the $n$th TT-core can be separated 
from the others, see, e.g., \cite{Dol2013b,KresSteinUsh2014}.  
By defining the so-called frame matrix
	\begin{equation}\label{eqn:frame_01}
	\BF{X}^{\neq n} = 
	\left(\BF{G}^{<n}_{(n)}\right)^\Trps
		\otimes \BF{I}_{I_n} \otimes
		\left(\BF{G}^{> n}_{(1)}\right)^\Trps 
	\in\BB{R}^{I_1I_2\cdots I_N\times R_{n-1}I_nR_n}
	\end{equation}
for $n=1,\ldots,N$, it is simplified as 
	\begin{equation} \label{eqn_vectorize_3} 
	\text{vec}(\ten{X}) = \BF{X}^{\neq n} \text{vec}(\ten{G}^{(n)}). 
	\end{equation}
We can obtain a similar expression for the vectorization as in \eqref{eqn_vectorize_3}, 
where two neighboring TT-cores are separated from the other TT-cores as 
(see, e.g., \cite{KresSteinUsh2014})
	\begin{equation*} 
	\text{vec}\left(\ten{X}\right)
	= \BF{X}^{\neq n,n+1}\text{vec}\left(\ten{G}^{(n)}\modcon\ten{G}^{(n+1)} \right), 
	\end{equation*}	
where 
	\begin{equation}\label{eqn:frame_02}
	\BF{X}^{\neq n,n+1} = 
		\left(\BF{G}^{<n}_{(n)}\right)^\Trps
		\otimes \BF{I}_{I_n} 
		\otimes \BF{I}_{I_{n+1}} 
		\otimes \left(\BF{G}^{> n+1}_{(1)}\right)^\Trps 
	\in\BB{R}^{I_1I_2\cdots I_N\times R_{n-1}I_nI_{n+1}R_{n+1}}
	\end{equation}
for $n=1,2,\ldots,N-1.$

\subsubsection{Orthogonalization of core tensors}

Left- or right-orthogonalization of 3rd-order TT-cores $\ten{G}^{(n)}$ of TT decomposition is defined
in the following way, which helps improve convergence of an iterative method using 
TT decomposition and reduces computational costs \cite{Holtz2012}.  

\begin{defn}[Left- or right-orthogonality, \cite{Holtz2011}]
For a fixed $n=1,\ldots,N,$ the $n$th TT-core $\ten{G}^{(n)} \in\BB{R}^{R_{n-1}\times I_n\times R_n}$ 
is called left-orthogonal if 
	\begin{equation}\label{eqn:orth_left}
	\BF{G}^{(n)}_{(3)}  \BF{G}^{(n)\Trps}_{(3)}= 
	\BF{I}_{R_n},
	\end{equation}
and right-orthogonal if 
	\begin{equation}\label{eqn:orth_right}
	\BF{G}^{(n)}_{(1)}\BF{G}^{(n)\Trps}_{(1)} =
	\BF{I}_{R_{n-1}}. 
	\end{equation}
\end{defn}

During iterations in numerical algorithms, the TT-cores are 
kept either left- or right-orthogonal by SVD or QR decomposition \cite{Holtz2012}. 
For a fixed $n$, if all the TT-cores on the left, i.e., $\ten{G}^{(m)},m=1,\ldots,n-1,$ are 
left-orthogonal, then $\BF{G}^{<n}_{(n)}\in\BB{R}^{R_n \times I_1I_2\cdots I_{n-1}}$ 
has orthonormal rows. In the same way, if all the TT-cores on the right, 
i.e., $\ten{G}^{(m)},m=n+1,\ldots,N,$ are right-orthogonal, then 
$\BF{G}^{>n}_{(1)}\in\BB{R}^{R_n\times I_{n+1}I_{n+2}\cdots I_N}$ 
has orthonormal rows. The proof of the orthonormality of $\BF{G}^{<n}_{(n)}$ and 
$\BF{G}^{>n}_{(1)}$ can be given by using the tensor operations described 
in the previous section as follows.

\begin{prop} Let $n=1,2,\ldots,N$ be fixed. 
	\begin{enumerate}
	\item[(a)] If $\ten{G}^{(1)},\ten{G}^{(2)}, \ldots,\ten{G}^{(n-1)}$ are left-orthogonal, 
			then $\BF{G}^{<n}_{(n)}$ has orthonormal rows.
	\item[(b)] If $\ten{G}^{(n+1)},\ten{G}^{(n+2)},\ldots,\ten{G}^{(N)}$ are right-orthogonal, 
			then $\BF{G}^{>n}_{(1)}$ has orthonormal rows. 
	\end{enumerate}
\end{prop}
\begin{proof}
We will prove (a) by induction.  (b) can be proved similarly. If $n=1$, then 
$\BF{G}^{<n}_{(n)} = \ten{G}^{<1} = 1$, so we have 
$\BF{G}^{<n}_{(n)} (\BF{G}^{<n}_{(n)})^\Trps = 1 = \BF{I}_{R_0}$ for $n=1$. 
Next, we suppose that $\BF{G}^{<k}_{(k)} (\BF{G}^{<k}_{(k)})^\Trps = \BF{I}_{R_{k-1}}$ 
under the left-orthogonality of $\ten{G}^{(1)},\ldots,\ten{G}^{(k-1)}$. 
Since $\ten{G}^{<k+1} = \ten{G}^{<k} \modcon \ten{G}^{(k)}$ for $k\leq N$, 
we can get the recursive expression (see also
Proposition~\ref{prop_modcon}(f) and   Table~\ref{Table:notation_TT_global_n_recursive})
	\begin{equation}
	\BF{G}^{<k+1}_{(k+1)} = \BF{G}^{(k)}_{(3)} \left( \BF{G}^{<k}_{(k)} \otimes \BF{I}_{I_k} \right). 
	\end{equation}
Since $\BF{G}^{<k}_{(k)} (\BF{G}^{<k}_{(k)})^\Trps = \BF{I}_{R_{k-1}}$, we have 
	$$
	\BF{G}^{<k+1}_{(k+1)} \left( \BF{G}^{<k+1}_{(k+1)} \right)^\Trps = 
	\BF{G}^{(k)}_{(3)} 
	\left( \BF{G}^{<k}_{(k)} (\BF{G}^{<k}_{(k)})^\Trps \otimes \BF{I}_{I_k} \right)
	\BF{G}^{(k)\Trps}_{(3)} 
	= \BF{G}^{(k)}_{(3)} \BF{G}^{(k)\Trps}_{(3)}. 
	$$
From the definition of the left-orthogonality \eqref{eqn:orth_left} of $\ten{G}^{(k)}$,  
we can conclude that $\BF{G}^{<k+1}_{(k+1)} ( \BF{G}^{<k+1}_{(k+1)} )^\Trps = \BF{I}_{R_k}$
under the left-orthogonality of $\ten{G}^{(1)},\ldots,\ten{G}^{(k)}$. 
\end{proof}

Note that the frame matrix $\BF{X}^{\neq n}$ \eqref{eqn:frame_01} will have 
orthonormal columns if both $\BF{G}^{<n}_{(n)}$ and $\BF{G}^{>n}_{(1)}$ 
have orthonormal rows. The same argument holds for the frame matrix 
$\BF{X}^{\neq n,n+1}$ \eqref{eqn:frame_02} with $\BF{G}^{<n}_{(n)}$ and 
$\BF{G}^{>n+1}_{(1)}$. 

Figure \ref{Fig:TT}(a),(b), and (c) show tensor network diagrams 
for the TT decomposition of a 4th-order tensor. The left-orthogonalized or 
right-orthogonalized core tensors are represented 
by half-filled circles. For example, 
Figure \ref{Fig:TT}(a) shows that the 
three among the four core tensors have been 
left-orthogonalized. 

\begin{figure}
\centering
\begin{tabular}{ccc}
\includegraphics[width=4cm]{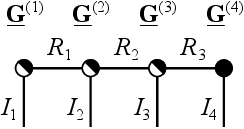}
&
\includegraphics[width=4cm]{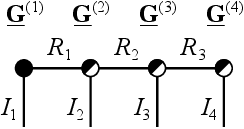}
&
\includegraphics[width=4cm]{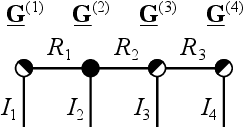}
\\
(a)&(b)&(c)
\end{tabular}
\caption{\label{Fig:TT}Tensor network diagrams for TT decomposition 
of a 4th-order tensor $\ten{X}\in\BB{R}^{I_1\times I_2\times I_3\times I_4}$. 
In each figure, three among the four TT-cores have been
(a) left-orthogonalized, (b) right-orthogonalized, 
or (c) left- or right-orthogonalized.}
\end{figure}

\subsubsection{Uniqueness of TT-ranks and minimal TT decomposition}

Note that the TT-ranks $R_1,\ldots,R_{N-1}$ in the TT decomposition 
\eqref{eqn:TTcontract} may not be unique, and depend on the decomposition 
rather than the tensor $\ten{X}$ itself. 

Mathematical (geometric) properties of TT decomposition
are closely related to the so-called separation ranks \cite{Holtz2011}. 
The $n$th separation rank, $S_n$, of a tensor 
$\ten{X}\in\BB{R}^{I_1\times\cdots\times I_N}$ is defined as the rank 
of the $n$th canonical unfolding (i.e., mode-($1,\ldots,n$) matricization) $\BF{X}_{\unfoldto{n}}$, i.e., 
	$$
	S_n = \text{rank}\left(\BF{X}_{\unfoldto{n}}\right). 
	$$
By looking at the canonical unfolding of the TT decomposition \eqref{eqn:TTcontract}, 
we can show the following relationship between the separation ranks and the TT-ranks:
	$$
	S_n\leq R_n
	$$
for $n=1,\ldots,N-1$. 

The TT decomposition \eqref{eqn:TTcontract} is called 
minimal or fulfilling the full-rank condition \cite{Holtz2011} if all 
TT-cores have full left and right ranks, i.e., 
	$$
	R_n = \text{rank}\left( \BF{G}^{(n)}_{(3)} \right)
	\quad
	\text{and}
	\quad
	R_{n-1} = \text{rank}\left( \BF{G}^{(n)}_{(1)} \right), 
	$$
for all $n=1,\ldots,N$. 
Holtz, Rohwedder, and Schneider \cite{Holtz2011} proved that 
the TT-ranks of minimal TT decompositions for a tensor $\ten{X}$ are unique, 
that is, if $\ten{X}$ admits for a minimal TT decomposition
with TT-ranks $R_n$, then it holds that
	$$
	R_n = S_n
	$$
for $n=1,\ldots,N-1$. Therefore, the TT-ranks of a tensor $\ten{X}$ can be defined as 
the TT-ranks of a minimal TT decomposition of $\ten{X}$, i.e., $R_n=\text{rank}(\BF{X}_{\unfoldto{n}})$.

Let $\BB{T}_\text{TT}(R_1,\ldots,R_{N-1})$ denote the set of tensors of TT-ranks 
bounded by $(R_1,\ldots,R_{N-1})$. It has been shown that 
$\BB{T}_\text{TT}(R_1,\ldots,R_{N-1})$ is closed \cite{FalHac2012,Hac2012}, 
which implies that there exists a best TT-ranks-$(R_1,\ldots,R_{N-1})$ 
approximation of any tensor. 
Oseledets \cite{Ose2011} proved that the TT-SVD algorithm proposed in \cite{Ose2011}
can return a quasi-optimal TT approximation of a given tensor $\ten{Y}\in\BB{R}^{I_1\times 
\cdots \times I_N}$ in principle. That is, for a tensor $\ten{Y}$, the TT-SVD algorithm returns
$\widehat{\ten{X}}\in\BB{T}_\text{TT}(R_1,\ldots,R_{N-1})$ such that
	$$
	\left\| \ten{Y}-\widehat{\ten{X}} \right\|_\text{F}
	\leq \sqrt{N-1} \cdot
	\min_{\ten{X}
	\in\BB{T}_\text{TT}(R_1,\ldots,R_{N-1})} 
	\left\| \ten{Y} - \ten{X} \right\|_\text{F}, 
	$$
where $\|\cdot\|_\RM{F}$ is the Frobenius norm. Moreover, 
the TT-SVD algorithm yields a minimal TT decomposition 
of a tensor $\ten{Y}$ \cite{Holtz2011,Ose2011}. 

Unlike the Tucker decomposition, the TT decomposition 
requires a storage cost of $\CL{O}(NIR^2)$, for 
$I=\max(I_n)$ and $R=\max(R_n)$, that is linear with the order $N$. 
Many algorithms based on TT decomposition 
such as the approximation, truncation, contraction, and solution to linear systems 
also have computational complexity linear with $N$. Further mathematical 
properties of TT-based algorithms will be developed in the next section.

\subsection{TT decomposition for large-scale vectors and matrices}
\label{sec_TT_vec_mat}

We will introduce variations of TT decomposition for representing 
large-scale vectors, matrices, and low-order tensors. 
Various representations for TT decomposition of vectors and matrices 
are summarized in Table~\ref{Table:notation_TT_vector_n_matrix}. 

\begin{table} 
\centering
\caption{\label{Table:notation_TT_vector_n_matrix}
Various representations for TT decomposition 
of a large-scale vector $\BF{x}\in\BB{R}^{I_1\cdots I_N}$
and a matrix $\BF{A}\in\BB{R}^{I_1\cdots I_N\times J_1\cdots J_N}$.}
\vspace{.5pc}
\begin{tabular}{ll}
\hline\\[-0.8pc]
\multicolumn{1}{c}{Vector TT}   &   \multicolumn{1}{c}{Matrix TT} \\
\hline
\multicolumn{2}{c}{Strong Kronecker product} \\
	$\displaystyle
	\BF{x} = 
	\widetilde{\BF{X}}^{(1)} \skron
	\widetilde{\BF{X}}^{(2)} \skron
	\cdots \skron
	\widetilde{\BF{X}}^{(N)}, 
	$
	&
	$\displaystyle
	\BF{A} = \widetilde{\BF{A}}^{(1)}  \skron 
	\widetilde{\BF{A}}^{(2)}  \skron 
	\cdots  \skron 
	\widetilde{\BF{A}}^{(N)}, 
	$
\\
		\quad 
		$\widetilde{\BF{X}}^{(n)} 
		= \begin{bmatrix}\BF{x}_{r_{n-1},r_n}^{(n)}\end{bmatrix}
		\in\BB{R}^{R_{n-1}I_n\times R_n}$ 
	&
		\quad 
		$\displaystyle 
		\widetilde{\BF{A}}^{(n)} = \begin{bmatrix}
		\BF{A}_{r^A_{n-1},r^A_n}^{(n)}\end{bmatrix}
		\in\BB{R}^{R^A_{n-1}I_n\times R^A_nJ_n}$ 
\\
		\quad 
		with each block $\BF{x}_{r_{n-1},r_n}^{(n)}
		\in\BB{R}^{I_n}$
	& 
		\quad 
		with each block $\BF{A}_{r^A_{n-1},r^A_n}^{(n)}
		\in\BB{R}^{I_n\times J_n}$
\\
\hline 
\multicolumn{2}{c}{Kronecker product} \\
	$\displaystyle
	\BF{x} = 
	\sum_{r_1,\ldots,r_{N-1}=1}^{R_1,\ldots,R_{N-1}} 
	\BF{x}^{(1)}_{1,r_1} \otimes
	\BF{x}^{(2)}_{r_1,r_2} \otimes
	\cdots \otimes
	\BF{x}^{(N)}_{r_{N-1},1}
	$
	&
	$\displaystyle 
	\BF{A} = 
	\sum_{r^A_1,\ldots,r^A_{N-1}=1}^{R^A_1,\ldots,R^A_{N-1}} 
	\BF{A}_{1,r^A_1}^{(1)} \otimes 
	\BF{A}_{r^A_1,r^A_2}^{(2)}
	\otimes \cdots \otimes 
	\BF{A}_{r^A_{N-1},1}^{(N)}$
\\
\hline
\multicolumn{2}{c}{Contracted product} \\
	$\displaystyle
	\ten{X} = \ten{X}^{(1)} \modcon \ten{X}^{(2)} \modcon \cdots 
	\modcon \ten{X}^{(N)}
	$
	&
	$\displaystyle
	\ten{A} = \ten{A}^{(1)} \modcon \ten{A}^{(2)} \modcon \cdots 
	\modcon \ten{A}^{(N)}
	$
\\
\hline
\multicolumn{2}{c}{Outer product} \\
	$\displaystyle 
	\ten{X} = 	
		\sum_{r_1,\ldots,r_{N-1}=1}^{R_1,\ldots,R_{N-1}} 
		\BF{x}_{1,r_1}^{(1)}\circ \BF{x}_{r_1,r_2}^{(2)}
		\circ \cdots \circ \BF{x}_{r_{N-1},1}^{(N)}
	$
	&
	$\displaystyle
	\ten{A} = 
		\sum_{r^A_1,\ldots,r^A_{N-1}=1}^{R^A_1,\ldots,R^A_{N-1}} 
		\BF{A}_{1,r^A_1}^{(1)} \circ 
		\BF{A}_{r^A_1,r^A_2}^{(2)}
		\circ \cdots \circ 
		\BF{A}_{r^A_{N-1},1}^{(N)}
	$
\\	
\hline
\multicolumn{2}{c}{Scalar product} \\
	$\displaystyle
	x_{i_1,\ldots,i_N} = \BF{x}(\overline{i_1\cdots i_N}) = 
	$
	&
	$\displaystyle
	a_{i_1,j_1,i_2,j_2,\ldots,i_N,j_N} = \BF{A}(\overline{i_1\cdots i_N},\overline{j_1\cdots j_N})
	= 
	$
\\
		\quad 
		$\displaystyle
		\sum_{r_1,\ldots,r_{N-1}=1}^{R_1,\ldots,R_{N-1}}
		x^{(1)}_{1,i_1,r_1} x^{(2)}_{r_1,i_2,r_2}\cdots x^{(N)}_{r_{N-1},i_N,1}
		$
	& 
		\quad
		$\displaystyle
		\sum_{r^A_1,\ldots,r^A_{N-1}=1}^{R^A_1,\ldots,R^A_{N-1}} 
		a^{(1)}_{1,i_1,j_1,r^A_1} a^{(2)}_{r^A_1,i_2,j_2,r^A_2} \cdots a^{(N)}_{r^A_{N-1},i_N,j_N,1}
		$
\\
\hline
\multicolumn{2}{c}{Matrix product} \\
	$\displaystyle 
	x_{i_1,\ldots,i_N} = 
	\BF{X}^{(1)}_{i_1} 
	\BF{X}^{(2)}_{i_2}
	\cdots 	\BF{X}^{(N)}_{i_N}
	$
	&
	$\displaystyle 
	a_{i_1,j_1,\ldots,i_N,j_N} = 
	\BF{A}^{(1)}_{i_1,j_1}
	\BF{A}^{(2)}_{i_2,j_2}
	\cdots 
	\BF{A}^{(N)}_{i_N,j_N}
	$
\\[0.2pc]
\hline 
\multicolumn{2}{c}{Vector/matrix representation} \\
	$\displaystyle 
	\BF{x}
		= \prod_{n=N}^1
		\left( \BF{I}_{I_1I_2\cdots I_{n-1}} \otimes \BF{X}^{(n)\Trps}_{(1)} \right)
	$
	&
	$\displaystyle
	\BF{A} = \prod_{n=1}^N 
	\left( \BF{I}_{J_1J_2\cdots J_{n-1}} \otimes \BF{A}^{(n)}_{\unfoldto{2}} 
	\otimes \BF{I}_{I_{n+1}I_{n+2}\cdots I_N} \right)
	$ 
\\
	$\displaystyle 
	\BF{x} = 
	\prod_{n=1}^N \left( \BF{X}^{(n)\Trps}_{(3)} \otimes \BF{I}_{I_{n+1}\cdots I_N} \right)
	$
	&
\\
	$\displaystyle 
	\BF{x} = \left( (\BF{X}^{<n}_{(n)} )^\Trps \otimes \BF{I}_{I_n}\otimes
		(\BF{X}^{>n}_{(1)} )^\Trps\right) 
		\text{vec}\left( \ten{X}^{(n)} \right)
	$
	&	
\\
\hline
\end{tabular}
\end{table}

A large-scale vector $\BF{x}$ of length $I_1I_2\cdots I_N$ 
can be considered to be reshaped into a higher-order tensor $\ten{X}\in\BB{R}^{I_1\times I_2\times \cdots \times I_N}$
and approximately represented by TT decomposition as
	\begin{equation} \label{eqn_vectorTT_tensor_modcon}
	\ten{X} = \ten{X}^{(1)} \modcon \ten{X}^{(2)} \modcon \cdots 
	\modcon \ten{X}^{(N)}, 
	\end{equation}
where $\ten{X}^{(n)} \in\BB{R}^{R_{n-1}\times I_n\times R_n}$ are 
3rd-order TT-cores. 
Vectorization of the above TT decomposition
yields the TT decomposition for $\BF{x}$, expressed as a sum of Kronecker products (e.g., see \eqref{eqn:TTouter})
	\begin{equation} \label{eqn_TT_vector_kron}
	\BF{x} = 
	\sum_{r_1=1}^{R_1}\sum_{r_2=1}^{R_2}\cdots
	\sum_{r_{N-1}=1}^{R_{N-1}}
	\BF{x}^{(1)}_{1,r_1} \otimes
	\BF{x}^{(2)}_{r_1,r_2} \otimes
	\cdots \otimes
	\BF{x}^{(N-1)}_{r_{N-2},r_{N-1}} \otimes
	\BF{x}^{(N)}_{r_{N-1},1}, 
	\end{equation}
where $\BF{x}^{(n)}_{r_{n-1},r_n} = \ten{X}^{(n)}(r_{n-1},:,r_n) \in\BB{R}^{I_n}$
for all $r_{n-1},r_n,$ and $n$. 
The above form can be compactly represented as
strong Kronecker products of block matrices
	\begin{equation} \label{eqn_TT_vector_skron}
	\BF{x} = 
	\widetilde{\BF{X}}^{(1)} \skron
	\widetilde{\BF{X}}^{(2)} \skron
	\cdots \skron
	\widetilde{\BF{X}}^{(N)}, 
	\end{equation}	
where $\widetilde{\BF{X}}^{(n)} = [\BF{x}^{(n)}_{r_{n-1}, r_n}] \in\BB{R}^{R_{n-1}I_n\times R_n}$ are block matrices 
partitioned with $\BF{x}^{(n)}_{r_{n-1},r_n}\in\BB{R}^{I_n}$
as
	$$
	\widetilde{\BF{X}}^{(n)}
	= \begin{bmatrix}
	\BF{x}^{(n)}_{1,1} & \cdots 
		& \BF{x}^{(n)}_{1,R_n}\\
	\vdots & & \vdots\\
	\BF{x}^{(n)}_{R_{n-1},1} & \cdots 
		& \BF{x}^{(n)}_{R_{n-1},R_n}
	\end{bmatrix}
	\in\BB{R}^{R_{n-1}I_n\times R_n}. 
	$$

TT decomposition can also be extended to a representation for 
linear operators and matrices \cite{Ose2010}. 
Suppose that a large-scale matrix $\BF{A}$
of size $I_1I_2\cdots I_N\times J_1J_2\cdots J_N$
is reshaped and permuted into a higher-order tensor 
$\ten{A}$ of size $I_1\times J_1\times I_2\times J_2
\times\cdots\times I_N\times J_N$.
Similarly to \eqref{eqn:TTcontract}, a TT decomposition for $\ten{A}$
can be represented as contracted products
	\begin{equation} \label{eqn_matrixTT_tensor_modcon}
	\ten{A} = 
	\ten{A}^{(1)} \modcon \ten{A}^{(2)} \modcon \cdots \modcon \ten{A}^{(N)}, 
	\end{equation}
where $\ten{A}^{(n)}\in\BB{R}^{R^A_{n-1}\times I_n\times J_n\times R^A_n}$	
are 4th-order TT-cores and $R^A_1,\ldots,R^A_{N-1}$ are TT-ranks. 
We assume that $R^A_0=R^A_N=1$. The TT decomposition for 
the tensor $\ten{A}$ in \eqref{eqn_matrixTT_tensor_modcon} can be 
equivalently expressed as outer products of slice matrices 
	$$
	\ten{A} = 
	\sum_{r^A_1=1}^{R^A_1}\sum_{r^A_2=1}^{R^A_2} \cdots 
	\sum_{r^A_{N-1}=1}^{R^A_{N-1}} 
	\BF{A}_{1,r^A_1}^{(1)} \circ 
	\BF{A}_{r^A_1,r^A_2}^{(2)}
	\circ \cdots \circ 
	\BF{A}_{r^A_{N-2},r^A_{N-1}}^{(N-1)}
	\circ 
	\BF{A}_{r^A_{N-1},1}^{(N)}, 
	$$
where $\BF{A}^{(n)}_{r^A_{n-1},r^A_n} = \ten{A}^{(n)}(r^A_{n-1},:,:,r^A_n) \in\BB{R}^{I_n\times J_n}$
is a slice of the 4th-order core tensor $\ten{A}^{(n)}\in
\BB{R}^{R_{n-1}\times I_n\times J_n\times R_n}$. 
We can derive that TT decomposition for the matrix 
$\BF{A}\in\BB{R}^{I_1\cdots I_N\times J_1\cdots J_N}$ can be
 represented as a sum of Kronecker products 
	\begin{equation} \label{eqn_TT_matrix_kron}
	\BF{A} = 
	\sum_{r^A_1=1}^{R^A_1}\sum_{r^A_2=1}^{R^A_2} \cdots 
	\sum_{r^A_{N-1}=1}^{R^A_{N-1}} 
	\BF{A}_{1,r^A_1}^{(1)} \otimes 
	\BF{A}_{r^A_1,r^A_2}^{(2)}
	\otimes \cdots \otimes 
	\BF{A}_{r^A_{N-2},r^A_{N-1}}^{(N-1)}
	\otimes 
	\BF{A}_{r^A_{N-1},1}^{(N)}, 
	\end{equation}
where $\BF{A}^{(n)}_{r^A_{n-1},r^A_n} = \ten{A}^{(n)}(r^A_{n-1},:,:,r^A_n) \in\BB{R}^{I_n\times J_n}$. 
The representation in \eqref{eqn_TT_matrix_kron} can be equivalently 
expressed as strong Kronecker product of block matrices 
	\begin{equation} \label{eqn_TT_matrix_skron}
	\BF{A} = \widetilde{\BF{A}}^{(1)}  \skron 
	\widetilde{\BF{A}}^{(2)}  \skron 
	\cdots  \skron 
	\widetilde{\BF{A}}^{(N)}, 
	\end{equation}
where $\widetilde{\BF{A}}^{(n)} = [\BF{A}^{(n)}_{r^A_{n-1},r^A_n}] 
\in\BB{R}^{R^A_{n-1}I_n\times R^A_nJ_n}$
are block matrices partitioned with the slice matrices $\BF{A}^{(n)}_{r^A_{n-1},r^A_n} \in\BB{R}^{I_n\times J_n}$ as 
	$$
	\widetilde{\BF{A}}^{(n)} = 
	\begin{bmatrix}
	\BF{A}^{(n)}_{1,1}&\cdots&\BF{A}^{(n)}_{1,R^A_n}\\
	\vdots&&\vdots\\
	\BF{A}^{(n)}_{R^A_{n-1},1}&\cdots&\BF{A}^{(n)}_{R^A_{n-1},R^A_n}
	\end{bmatrix}
	\in\BB{R}^{R^A_{n-1}I_n\times R^A_nJ_n}. 
	$$
The strong Kronecker product representation \eqref{eqn_TT_matrix_skron} 
for TT decomposition is useful, especially for representing large-scale 
high-dimensional operators such as discrete Laplace operator and 
multilevel (hierarchical) Toeplitz, Hankel, circulant, banded diagonal matrices \cite{Kaz2010,KazKhoTyr2011}.

We call the TT decomposition in \eqref{eqn_TT_vector_kron} and 
\eqref{eqn_TT_vector_skron} for vectors as the vector TT decomposition, 
and the TT decomposition in \eqref{eqn_TT_matrix_kron} and 
\eqref{eqn_TT_matrix_skron} for matrices as the matrix TT decomposition. 
Note that the blocks $\BF{x}^{(n)}_{r_{n-1},r_n}\in\BB{R}^{I_n}$ and 
$\BF{A}^{(n)}_{r^A_{n-1},r^A_n}\in\BB{R}^{I_n\times J_n}$ in the 
expressions can be further generalized to higher-order tensors, which  
leads to TT decomposition for multilinear operators and higher-order tensors.

\section{Basic operations using TT decomposition}
\label{app:basic}


Based on the introduced TT decompositions, basic operations on 
large-scale vectors and matrices such as matrix-by-vector multiplication 
can be performed fast and conveniently. 
For basic algebraic operations on tensors represented by 
TT decomposition, it is important to perform all the operations based on 
the factor (core) tensors in TT decomposition and avoid the explicit 
calculation of the full tensors. In this section we present basic operations 
on vectors, matrices, and higher-order tensors represented 
by TT decomposition, in simple and very efficient  forms by using the notations 
introduced in the previous sections. 
The basic operations include addition, scalar multiplication, direct sum, 
Hadamard product, Kronecker product, 
contraction, matrix-by-vector product, matrix-by-matrix product, and quadratic form. 
Figure \ref{Fig:TTMatrix} illustrates tensor network diagrams
for the matrix TT decomposition, the matrix-by-vector multiplication, 
and the quadratic form, represented by TT decomposition.  

Note that such operations usually increase the TT-ranks, 
which requires truncation (rounding) in the following step \cite{Ose2011}. In addition,
the matrix-by-vector product and quadratic form are very important 
for computational algorithms in optimization problems such as 
linear equations \cite{DolSav2014,Holtz2012,OseDol2012}
and eigenvalue problems \cite{Dol2013b,KhoOse2010a,KresSteinUsh2014}. 
Tables~\ref{Table:basic_operations_TT} and \ref{Table:basic_operations_vecmatTT} 
summarize
the representations for the basic operations on the tensors represented by TT decomposition. 

\begin{figure}
\centering
\begin{tabular}{ccc}
\includegraphics[width=4.3cm]{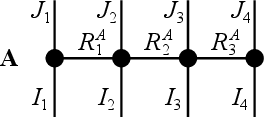}& 
\includegraphics[width=4.3cm]{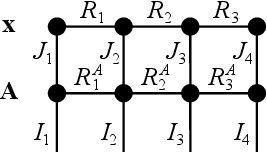}& 
\includegraphics[width=4.3cm]{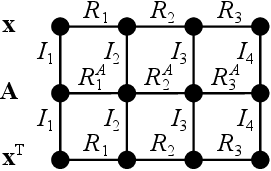}\\
(a) $\BF{A}$ & (b) $\BF{Ax}$ & (c) $\BF{x}^\Trps\BF{Ax}$
\end{tabular}
\caption{\label{Fig:TTMatrix}
(a) A matrix $\BF{A}\in\BB{R}^{I_1I_2I_3I_4\times J_1J_2J_3J_4}$ 
in matrix TT decomposition, 
(b) matrix-by-vector product $\BF{y}=\BF{Ax}$ in TT decomposition, 
and (c) quadratic form $\BF{x}^\Trps\BF{Ax}$ when 
$I_n=J_n,n=1,\ldots,4.$ }
\end{figure}

\begin{table} 
\centering
\caption{\label{Table:basic_operations_TT}
TT representations for basic operations on tensors represented by 
TT decomposition.}
\vspace{.5pc}
\begin{tabular}{lll}
\hline\\[-0.8pc]
Operation &  & TT-cores\\
\hline\\[-0.8pc]
\multicolumn{3}{l}{TT (global)*}\\
\multicolumn{3}{l}{
$\ten{Z}=\ten{X}+\ten{Y}
		 = \left(\BF{X}^{(1)} \boxplus \BF{Y}^{(1)}\right) \modcon
		\left(\ten{X}^{(2)} \boxplus \ten{Y}^{(2)}\right) \modcon\cdots\modcon
		\left(\BF{X}^{(n)} \boxplus \BF{Y}^{(n)}\right) $} \\
	& Ten & $\ten{Z}^{(n)} = \ten{X}^{(n)} \boxplus \ten{Y}^{(n)}$\\
	& Mat & $\BF{Z}^{(n)}_{i_n} = \BF{X}^{(n)}_{i_n} \oplus 
					   \BF{Y}^{(n)}_{i_n}$\\
	& Vec & $\BF{z}^{(n)}_{s_{n-1},s_n}=\BF{x}^{(n)}_{s_{n-1},s_n}$; 
			$\BF{y}^{(n)}_{s_{n-1}-R_{n-1}^X,s_n-R_n^X}$;
			$\BF{0}_{I_n}$\\
		
\multicolumn{3}{l}{
$\ten{Z}=\ten{X} \oplus \ten{Y}
		 = \left(\BF{X}^{(1)} \oplus \BF{Y}^{(1)}\right) \modcon
		\left(\ten{X}^{(2)} \oplus \ten{Y}^{(2)}\right) \modcon\cdots\modcon
		\left(\BF{X}^{(n)} \oplus \BF{Y}^{(n)}\right) $} \\
	& Ten & $\ten{Z}^{(n)} = \ten{X}^{(n)} \oplus \ten{Y}^{(n)}$\\
	& Mat & 
		$\displaystyle 
		\BF{Z}^{(n)}_{k_n}=\begin{cases}
		\BF{X}^{(n)}_{k_n} \oplus \BF{0}_{R^Y_{n-1}\times R^Y_n} & 
		\text{if } 1\leq k_n\leq I_n\\ 
		\BF{0}_{R^X_{n-1}\times R^X_n} \oplus \BF{Y}^{(n)}_{k_n-I_n} & 
		\text{if } I_n < k_n \leq I_n+J_n
		\end{cases}$\\
	& Vec & $\BF{z}^{(n)}_{s_{n-1},s_n}=
		\BF{x}^{(n)}_{s_{n-1},s_n} \oplus \BF{0}_{J_n}; 
		\BF{0}_{I_n} \oplus \BF{y}^{(n)}_{s_{n-1}-R_{n-1}^X,s_n-R_n^X}; 
		\BF{0}_{I_n+J_n}$\\
		
\multicolumn{3}{l}{
$\ten{Z}=\ten{X}\circledast\ten{Y}
		 = \left(\ten{X}^{(1)} \boxtimes \ten{Y}^{(1)}\right) \modcon
		\left(\ten{X}^{(2)} \boxtimes \ten{Y}^{(2)}\right) \modcon\cdots\modcon
		\left(\ten{X}^{(n)} \boxtimes \ten{Y}^{(n)}\right) $} \\
	& Ten & $\ten{Z}^{(n)}=\ten{X}^{(n)} \boxtimes \ten{Y}^{(n)}$\\
	& Mat & $\BF{Z}^{(n)}_{i_n} = \BF{X}^{(n)}_{i_n} \otimes \BF{Y}^{(n)}_{i_n}$\\
	& Vec & $\BF{z}^{(n)}_{s_{n-1},s_n}=
		\BF{x}^{(n)}_{r^X_{n-1},r^X_n}\circledast
		\BF{y}^{(n)}_{r^Y_{n-1},r^Y_n}$\quad ($s_n = \overline{r^X_nr^Y_n}$)\\
		
\multicolumn{3}{l}{
$\ten{Z}=\ten{X} \otimes \ten{Y}
		 = \left(\ten{X}^{(1)} \otimes \ten{Y}^{(1)}\right) \modcon
		\left(\ten{X}^{(2)} \otimes \ten{Y}^{(2)}\right) \modcon\cdots\modcon
		\left(\ten{X}^{(n)} \otimes \ten{Y}^{(n)}\right) $} \\
	& Ten & $\ten{Z}^{(n)}=\ten{X}^{(n)} \otimes \ten{Y}^{(n)}$\\
	& Mat & $\BF{Z}^{(n)}_{k_n} = \BF{X}^{(n)}_{i_n} \otimes \BF{Y}^{(n)}_{j_n}$\quad ($k_n=\overline{i_nj_n}$)\\
	& Vec & $\BF{z}^{(n)}_{s_{n-1},s_n}=
		\BF{x}^{(n)}_{r^X_{n-1},r^X_n}\otimes
		\BF{y}^{(n)}_{r^Y_{n-1},r^Y_n}$\quad ($s_n = \overline{r^X_nr^Y_n}$)\\

\multicolumn{3}{l}{
$\ten{Z}=\ten{X} \circ \ten{Y}
		 = \ten{X}^{(1)} \modcon \cdots \modcon \ten{X}^{(N)} 
		\modcon \ten{Y}^{(1)}  \modcon \cdots \modcon \ten{Y}^{(N)}$} \\
	& Ten & $\ten{Z}^{(n)}=\ten{X}^{(n)}$  $(n\leq N)$; $\ten{Y}^{(n-N)}$ $(n>N)$\\
	& Mat & $\BF{Z}^{(n)}_{i_n} = \BF{X}^{(n)}_{i_n}$ $(n\leq N)$;  $\BF{Y}^{(n-N)}_{i_n}$ $(n>N)$\\
	& Vec & $\BF{z}^{(n)}_{s_{n-1},s_n}=
		\BF{x}^{(n)}_{s_{n-1},s_n}$  $(n\leq N)$; 
		$\BF{y}^{(n-N)}_{s_{n-1},s_n}$ $(n>N)$\\

\multicolumn{3}{l}{
$\ten{Z} = \ten{X} \times_n \BF{A}
		 = \ten{X}^{(1)} \modcon \cdots \modcon \ten{X}^{(n-1)}
		\modcon \left(\ten{X}^{(n)}\times_2\BF{A} \right) \modcon 
		\ten{X}^{(n+1)}  \modcon \cdots \modcon \ten{X}^{(N)}$} \\
	& Ten & $\ten{Z}^{(m)}=\ten{X}^{(m)}$ $(m\neq n)$; 
		      $\ten{X}^{(m)} \times_2 \BF{A}$ $(m=n)$\\
	& Mat &  $\BF{Z}^{(m)}_{i_m} = \BF{X}^{(m)}_{i_m}$ $(m\neq n)$;  
			$\ten{X}^{(m)} \overline{\times}_2 \BF{a}_{i_m,:}$ $(m=n)$
			\\
	& Vec & $\BF{z}^{(m)}_{s_{m-1},s_m}=
		\BF{x}^{(m)}_{s_{m-1},s_m}$  $(m\neq n)$; 
		$\BF{A}\BF{x}^{(m)}_{s_{m-1},s_m}$ $(m=n)$\\

\multicolumn{3}{l}{
$\ten{Z}=\ten{X} \modcon \ten{Y} 
		 = \ten{X}^{(1)} \modcon \cdots \modcon \ten{X}^{(N-1)} 
		\modcon \left( \BF{X}^{(N)}\BF{Y}^{(1)} \modcon 
		\ten{Y}^{(2)} \right) \modcon \ten{Y}^{(3)} \modcon \cdots \modcon \ten{Y}^{(N)}$} \\
	& Ten & $\ten{Z}^{(n)}=\ten{X}^{(n)}$  $(n< N)$; 
			$\BF{X}^{(N)}\BF{Y}^{(1)}\modcon\ten{Y}^{(2)}$; 
			$\ten{Y}^{(n-N+2)}$ $(N<n)$\\
	& Mat & $\BF{Z}^{(n)}_{i_n} = \BF{X}^{(n)}_{i_n}$ $(n< N)$;  
			$\BF{X}^{(N)}\BF{Y}^{(1)}\BF{Y}^{(2)}_{i_n}$; 
			$\BF{Y}^{(n-N+2)}_{i_n}$ $(N<n)$\\
	& Vec & $\BF{z}^{(n)}_{s_{n-1},s_n}=
		\BF{x}^{(n)}_{s_{n-1},s_n}$  $(n< N)$; 
		$\BF{x}^{(N)\Trps}_{s_{n-1}} \BF{Y}^{(1)} \BF{Y}^{(2)}_{:,:,s_n}$; 
		$\BF{y}^{(n-N+2)}_{s_{n-1},s_n}$ $(N<n)$\\

$z=\left\langle \ten{X}, \ten{Y}\right\rangle$
	& \multicolumn{2}{l}{} \\
	& Ten & 
		$\BF{Z}^{(n)} = \left( \ten{X}^{(n)} \boxtimes \ten{Y}^{(n)} \right) \overline{\times}_2 \BF{1}_{I_n} $\\
	& Mat & $\BF{Z}^{(n)}= \sum_{i_n}\BF{X}^{(n)}_{i_n}\otimes \BF{Y}^{(n)}_{i_n}$\\
	& Vec & $z^{(n)}_{s_{n-1},s_n}=
		\left\langle \BF{x}^{(n)}_{r^X_{n-1},r^X_n}, 
		\BF{y}^{(n)}_{r^Y_{n-1},r^Y_n}\right\rangle$\quad ($s_n = \overline{r^X_nr^Y_n}$)
		\\
\hline
\multicolumn{3}{l}{* Ten = Tensors (cores), Mat = Matrices (slices), Vec = Vectors (fibers)}
\end{tabular}
\end{table}

\begin{table} 
\centering
\caption{\label{Table:basic_operations_vecmatTT}
TT representations for basic operations on tensors represented by 
vector TT and matrix TT decompositions.}
\vspace{.5pc}
\begin{tabular}{lll}
\hline\\[-0.8pc]
Operation &  & TT-cores\\
\hline\\[-0.8pc]
\multicolumn{3}{l}{Vector TT \& Matrix TT*}\\

\multicolumn{3}{l}{
$\BF{Z} = \BF{A} + \BF{B} 
	= \begin{bmatrix} \widetilde{\BF{A}}^{(1)}&\widetilde{\BF{B}}^{(1)}\end{bmatrix}
	\skron \begin{bmatrix} \widetilde{\BF{A}}^{(2)}&\BF{0}\\ \BF{0}&\widetilde{\BF{B}}^{(2)}\end{bmatrix}
	\skron\cdots \skron \begin{bmatrix} \widetilde{\BF{A}}^{(N-1)}&\BF{0}\\ \BF{0}&\widetilde{\BF{B}}^{(N-1)}\end{bmatrix}
	\skron \begin{bmatrix} \widetilde{\BF{A}}^{(N)}\\ \widetilde{\BF{B}}^{(N)}\end{bmatrix}
	$}\\
	& Blk & $\widetilde{\BF{Z}}^{(n)} = \widetilde{\BF{A}}^{(n)} \boxplus 
	\widetilde{\BF{B}}^{(n)}$\\
	& Mat & $\BF{Z}^{(n)}_{i_n,j_n} = \BF{A}^{(n)}_{i_n,j_n} \oplus \BF{B}^{(n)}_{i_n,j_n}$\\
	& Vec & $\BF{Z}^{(n)}_{s_{n-1},s_n}=
		\BF{A}^{(n)}_{s_{n-1},s_n}$; 
		$\BF{B}^{(n)}_{s_{n-1}-R^A_{n-1},s_n-R^A_n}$; $\BF{0}$\\



\multicolumn{3}{l}{
$\BF{Z} = \BF{A} \otimes \BF{B}
	= \widetilde{\BF{A}}^{(1)} \skron \cdots \skron \widetilde{\BF{A}}^{(N)} 
		\skron \widetilde{\BF{B}}^{(1)}  \skron \cdots \skron \widetilde{\BF{B}}^{(N)}$} \\
	& Blk & $\widetilde{\BF{Z}}^{(n)} = \widetilde{\BF{A}}^{(n)}$  $(n\leq N)$; $\widetilde{\BF{B}}^{(n-N)}$ $(n>N)$\\
	& Mat & $\BF{Z}^{(n)}_{i_n,j_n} = \BF{A}^{(n)}_{i_n,j_n}$ $(n\leq N)$;  $\BF{B}^{(n-N)}_{i_n,j_n}$ $(n>N)$\\
	& Vec & $\BF{Z}^{(n)}_{s_{n-1},s_n}=
		\BF{A}^{(n)}_{s_{n-1},s_n}$  $(n\leq N)$; 
		$\BF{B}^{(n-N)}_{s_{n-1},s_n}$ $(n>N)$\\


\multicolumn{3}{l}{	
$\ten{Z} = \ten{X} \times_n \BF{A}
	= \left( \widetilde{\ten{X}}^{(1)} \times_n
		\widetilde{\BF{A}}^{(1)} \right) \skron\cdots\skron
	\left( \widetilde{\ten{X}}^{(N)} \times_n
		\widetilde{\BF{A}}^{(N)} \right)	
	$}\\
	& Blk & 
		$\widetilde{\ten{Z}}^{(n)}=\widetilde{\ten{X}}^{(n)} \times_n
		\widetilde{\BF{A}}^{(n)}$\\
	& Mat & $\BF{Z}^{(n)}_{ijk} = \sum_{t} \BF{X}^{(n)}_{ijt} \otimes \BF{A}^{(n)}_{kt}$\\
	& Vec & $\ten{Z}^{(n)}_{s_{n-1},s_n}=
		\ten{X}^{(n)}_{r^X_{n-1},r^X_n} \times_n 
		\BF{A}^{(n)}_{r^A_{n-1},r^A_n}$\quad ($s_n = \overline{r^X_nr^A_n}$)
		\\

\multicolumn{3}{l}{	
$z=\BF{x}^\Trps\BF{y}=\left\langle \BF{x}, \BF{y}\right\rangle
	= \left( \widetilde{\BF{X}}^{(1)} \modcon
		\widetilde{\BF{Y}}^{(1)} \right) \skron\cdots\skron
	\left( \widetilde{\BF{X}}^{(N)} \modcon
		\widetilde{\BF{Y}}^{(N)} \right)	
	$}\\
	& Blk & 
		$\widetilde{\BF{Z}}^{(n)}=\widetilde{\BF{X}}^{(n)} \modcon
		\widetilde{\BF{Y}}^{(n)}$\\
	& Mat & $\BF{Z}^{(n)}= \sum_{i_n}\BF{X}^{(n)}_{i_n}\otimes \BF{Y}^{(n)}_{i_n}$\\
	& Vec & $z^{(n)}_{s_{n-1},s_n}=
		\left\langle \BF{x}^{(n)}_{r^X_{n-1},r^X_n}, 
		\BF{y}^{(n)}_{r^Y_{n-1},r^Y_n}\right\rangle$\quad ($s_n = \overline{r^X_nr^Y_n}$)
		\\

\multicolumn{3}{l}{			
$\BF{z}= \BF{Ax}
	= \left( \widetilde{\BF{A}}^{(1)} \modcon
		\widetilde{\BF{X}}^{(1)} \right) \skron\cdots\skron
	\left( \widetilde{\BF{A}}^{(N)} \modcon
		\widetilde{\BF{X}}^{(N)} \right)	
	$}\\
	& Blk & $\widetilde{\BF{Z}}^{(n)}=\widetilde{\BF{A}}^{(n)} \modcon \widetilde{\BF{X}}^{(n)}$\\
	& Mat & $\BF{Z}^{(n)}_{i_n} = \sum_{j_n} \BF{A}^{(n)}_{i_n,j_n}\otimes 
		\BF{X}^{(n)}_{j_n}$\\
	& Vec & 
		$\BF{z}^{(n)}_{s_{n-1},s_n}= \BF{A}^{(n)}_{r^A_{n-1},r^A_n} \BF{x}^{(n)}_{r_{n-1},r_n}$\quad ($s_n = \overline{r^A_nr_n}$)\\

\multicolumn{3}{l}{			
$\BF{Z}= \BF{AB}
	= \left( \widetilde{\BF{A}}^{(1)} \modcon
		\widetilde{\BF{B}}^{(1)} \right) \skron\cdots\skron
	\left( \widetilde{\BF{A}}^{(N)} \modcon
		\widetilde{\BF{B}}^{(N)} \right)	
	$}\\
	& Blk & $\widetilde{\BF{Z}}^{(n)}=\widetilde{\BF{A}}^{(n)} \modcon \widetilde{\BF{B}}^{(n)}$\\
	& Mat & $\BF{Z}^{(n)}_{i_n,j_n} = \sum_{k_n} \BF{A}^{(n)}_{i_n,k_n}\otimes 
		\BF{B}^{(n)}_{k_n,j_n}$\\
	& Vec & 
		$\BF{Z}^{(n)}_{s_{n-1},s_n}= \BF{A}^{(n)}_{r^A_{n-1},r^A_n} \BF{B}^{(n)}_{r^B_{n-1},r^B_n}$\quad ($s_n = \overline{r^A_nr^B_n}$)\\

\multicolumn{3}{l}{			
$z = \BF{x}^\Trps\BF{Ax} = \langle \BF{x},\BF{Ax}\rangle
	= \left( \widetilde{\BF{X}}^{(1)} \modcon \widetilde{\BF{A}}^{(1)} \modcon
		\widetilde{\BF{X}}^{(1)} \right) \skron\cdots\skron
	\left( \widetilde{\BF{X}}^{(N)} \modcon \widetilde{\BF{A}}^{(N)} \modcon
		\widetilde{\BF{X}}^{(N)} \right)	
	$}\\
	& Blk & $\widetilde{\BF{Z}}^{(n)}=\widetilde{\BF{X}}^{(n)} \modcon \widetilde{\BF{A}}^{(n)} \modcon \widetilde{\BF{X}}^{(n)}$\\
	& Mat & $\BF{Z}^{(n)} = \sum_{i_n}\sum_{j_n}
		\BF{X}^{(n)}_{i_n}\otimes \BF{A}^{(n)}_{i_n,j_n}\otimes 
		\BF{X}^{(n)}_{j_n}$\\
	& Vec & $z^{(n)}_{s_{n-1},s_n}=
		\left\langle \BF{x}^{(n)}_{r_{n-1}',r_n'}, 
		\BF{A}^{(n)}_{r^A_{n-1},r^A_n}\BF{x}^{(n)}_{r_{n-1},r_n}
		\right\rangle$\quad ($s_n = \overline{r_n'r^A_nr_n}$)\\
\hline
\multicolumn{3}{l}{* Blk = Block matrices, Mat = Matrices (slices), Vec = Vectors (fibers)}
\end{tabular}
\end{table}

\subsection{Addition and scalar multiplication}

Let 	
$\ten{X}=\ten{X}^{(1)}  
\modcon\cdots\modcon \ten{X}^{(N)} \in\BB{R}^{I_1\times\cdots\times I_N}$ and 
$\ten{Y}=\ten{Y}^{(1)} 
\modcon\cdots\modcon \ten{Y}^{(N)} \in\BB{R}^{I_1\times\cdots\times I_N}$ be TT tensors
with TT-ranks $\{R^X_n\}$ and $\{R^Y_n\}$. 
Under the assumption that the 1st and the $N$th TT-cores are matrices (i.e., 
2nd-order tensors,) 
the sum $\ten{Z}=\ten{X}+\ten{Y} \in\BB{R}^{I_1\times\cdots\times I_N}$ can be expressed 
by the TT decomposition (see, Proposition~\ref{prop_TT_algebra})
	\begin{equation}
	\ten{Z}=  
	\left( \BF{X}^{(1)}\boxplus \BF{Y}^{(1)} \right) \modcon
	\left( \ten{X}^{(2)}\boxplus \ten{Y}^{(2)} \right) \modcon
	\cdots\modcon
	\left( \BF{X}^{(N)}\boxplus \BF{Y}^{(N)} \right). 
	\end{equation}
Note that each TT-core of the sum $\ten{Z}=\ten{X}+\ten{Y}$
is written as the partial direct sum of the
 corresponding TT-cores. 
 
 Alternatively, each entry of the sum 
can be represented as products of slice matrices of TT-cores
	\begin{equation*}
	\begin{split}
	z_{i_1,i_2,\ldots,i_N}
	& =
	\left( \BF{X}^{(1)}_{i_1}\oplus \BF{Y}^{(1)}_{i_1} \right)
	\left( \BF{X}^{(2)}_{i_2} \oplus \BF{Y}^{(2)}_{i_2} \right)
	\cdots
	\left( \BF{X}^{(N)}_{i_N}\oplus \BF{Y}^{(N)}_{i_N} \right)
	\\
	& = 
	\begin{bmatrix} \BF{X}^{(1)}_{i_1} & \BF{Y}^{(1)}_{i_1} \end{bmatrix}
	\begin{bmatrix} \BF{X}^{(2)}_{i_2} & \\ & \BF{Y}^{(2)}_{i_2} \end{bmatrix}
	\cdots 
	\begin{bmatrix} \BF{X}^{(N)}_{i_N}\\ \BF{Y}^{(N)}_{i_N} \end{bmatrix}, 
	\end{split}
	\end{equation*}
where the first and the last factor matrices are assumed row and column vectors,
 respectively. The TT-ranks of the above TT decomposition for $\ten{X}+\ten{Y}$ 
are the sums, $\{R^X_n + R^Y_n\}$. 

On the other hand, multiplication of a TT tensor $\ten{X}$ with a scalar $c\in\BB{R}$
can be obtained by simply multiplying one core, e.g., $\ten{X}^{(1)}$, with $c$
as $c\ten{X}^{(1)}$. This does not increase the TT-ranks. 

We note that that the set of tensors with TT-ranks bounded by $\{R_n\}$ 
is not convex, since a linear combination 
$c\ten{X}+(1-c)\ten{Y}$ generally increases the TT-ranks, which 
may exceed $\{R_n\}$. 

\subsection{Direct sum}

Let 	
$\ten{X}=\ten{X}^{(1)} 
\modcon\cdots\modcon \ten{X}^{(N)} \in\BB{R}^{I_1\times\cdots\times I_N}$ and 
$\ten{Y}=\ten{Y}^{(1)} 
\modcon\cdots\modcon \ten{Y}^{(N)} \in\BB{R}^{J_1\times\cdots\times J_N}$ be TT tensors
with TT-ranks $\{R^X_n\}$ and $\{R^Y_n\}$. 
Under the assumption that the 1st and the $N$th TT-cores are matrices (i.e., 2nd-order tensors,) the direct sum $\ten{Z}=\ten{X} \oplus \ten{Y}$ can be expressed 
by the TT decomposition (see, Proposition~\ref{prop_TT_algebra})
	\begin{equation}
	\ten{Z}=  
	\left( \BF{X}^{(1)}\oplus \BF{Y}^{(1)} \right) \modcon
	\left( \ten{X}^{(2)}\oplus \ten{Y}^{(2)} \right) \modcon
	\cdots\modcon
	\left( \BF{X}^{(N)}\oplus \BF{Y}^{(N)} \right). 
	\end{equation}	
The TT-ranks of the above TT decomposition for $\ten{X}\oplus\ten{Y}$ are the sums, $\{R^X_n+R^Y_n\}$.

\subsection{Hadamard product}

The Hadamard (elementwise) product $\ten{Z}=\ten{X}\circledast\ten{Y}$ of 
two TT tensors $\ten{X}=\ten{X}^{(1)}\modcon \cdots \modcon\ten{X}^{(N)}
\in\BB{R}^{I_1\times\cdots\times I_N}$ 
and $\ten{Y}=\ten{Y}^{(1)} \modcon\cdots\modcon \ten{Y}^{(N)}
\in\BB{R}^{I_1\times\cdots\times I_N}$
can be expressed by TT decomposition as (see, Proposition~\ref{prop_TT_algebra})
	\begin{equation}
	\ten{Z}
	= \left( \ten{X}^{(1)} \boxtimes \ten{Y}^{(1)} \right)
	\modcon
	\left( \ten{X}^{(2)} \boxtimes \ten{Y}^{(2)} \right)
	\modcon\cdots\modcon
	\left( \ten{X}^{(N)} \boxtimes \ten{Y}^{(N)} \right). 
	\end{equation}
As an alternative representation, each entry can be written as products of slice matrices of TT-cores 
	$$
	z_{i_1,i_2,\ldots,i_N}
	= 
	\left( \BF{X}^{(1)}_{i_1} \otimes \BF{Y}^{(1)}_{i_1} \right)
	\left( \BF{X}^{(2)}_{i_2} \otimes \BF{Y}^{(2)}_{i_2} \right)
	\cdots
	\left( \BF{X}^{(N)}_{i_N} \otimes \BF{Y}^{(N)}_{i_N} \right). 
	$$
The TT-ranks for the above Hadamard product representation 
are the multiplications of individual ranks, $\{R^X_nR^Y_n\}$.

\subsection{Kronecker product}

The Kronecker product $\ten{Z} = \ten{X}\otimes \ten{Y}$
of two TT tensors $\ten{X}=\ten{X}^{(1)} 
\modcon\cdots\modcon \ten{X}^{(N)} \in\BB{R}^{I_1\times\cdots\times I_N}$ and 
$\ten{Y}=\ten{Y}^{(1)} 
\modcon\cdots\modcon \ten{Y}^{(N)} \in\BB{R}^{J_1\times\cdots\times J_N}$ 
with TT-ranks $\{R^X_n\}$ and $\{R^Y_n\}$ 
can be expressed by TT decomposition as  
(see, Proposition~\ref{prop_TT_algebra})
	\begin{equation}
	\ten{Z}=  
	\left( \ten{X}^{(1)}\otimes \ten{Y}^{(1)} \right) \modcon
	\left( \ten{X}^{(2)}\otimes \ten{Y}^{(2)} \right) \modcon
	\cdots\modcon
	\left( \ten{X}^{(N)}\otimes \ten{Y}^{(N)} \right). 
	\end{equation}	
The TT-ranks of the above TT decomposition for $\ten{X}\otimes\ten{Y}$ are the multiplications of individual ranks, $\{R^X_nR^Y_n\}$. 

\subsection{Full contraction: Inner product}

The contraction of two tensors $\ten{A}\in\BB{R}^{I_1\times I_2\times\cdots
\times I_N}$ and $\ten{B}\in\BB{R}^{I_1\times I_2\times\cdots
\times I_N}$ is defined by 
	\begin{equation*}
	\begin{split}
	\left\langle \ten{A}, \ten{B} 	\right\rangle
	& = \sum_{i_1=1}^{I_1}\sum_{i_2=1}^{I_2}\cdots\sum_{i_N=1}^{I_N}
	\ten{A}(i_1,i_2,\ldots,i_N)
	\ten{B}(i_1,i_2,\ldots,i_N). 
	\end{split}
	\end{equation*}

\begin{exmp}
The contraction of a TT tensor $\ten{X}=\ten{X}^{(1)}\modcon\cdots
\modcon\ten{X}^{(N)}$ with 
a rank-one tensor $\BF{u}^{(1)}\circ\cdots\circ\BF{u}^{(N)}$ 
can be simplified as
	\begin{equation}
	\begin{split}
	\left\langle \ten{X}, \BF{u}^{(1)}\circ\cdots\circ\BF{u}^{(N)} \right\rangle
	& = \left( \ten{X}^{(1)}\modcon\cdots
		\modcon\ten{X}^{(N)} \right)
		\overline{\times}_1 \BF{u}^{(1)}
		\cdots
		\overline{\times}_N \BF{u}^{(N)}\\
	& = \left(\ten{X}^{(1)}\overline{\times}_2\BF{u}^{(1)}\right)
		\left(\ten{X}^{(2)}\overline{\times}_2\BF{u}^{(2)}\right)
		\cdots
		\left(\ten{X}^{(N)}\overline{\times}_2\BF{u}^{(N)}\right). 
	\end{split}
	\end{equation}
\end{exmp}


The full contraction of two TT tensors 
$\ten{X}=\ten{X}^{(1)}\modcon\cdots \modcon\ten{X}^{(N)}$ and
$\ten{Y}=\ten{Y}^{(1)}\modcon\cdots \modcon\ten{Y}^{(N)}$
can be calculated by combining the Hadamard product and 
the contraction with the rank-one tensor 
$\BF{1}_{I_1}\circ\cdots\circ\BF{1}_{I_N}$ as 
	$$
	\left\langle \ten{X}, \ten{Y} \right\rangle = 
	\left\langle \ten{X} \circledast \ten{Y},  \BF{1}_{I_1} \circ\cdots\circ\BF{1}_{I_N} \right\rangle = \BF{Z}_1\cdots \BF{Z}_N,
	$$
where 
	\begin{equation}
	\begin{split}
	\BF{Z}_n 
	& = (\ten{X}^{(n)} \boxtimes \ten{Y}^{(n)}) \overline{\times}_2 \BF{1}_{I_n} \\
	& = \sum_{i_n=1}^{I_n} \BF{X}^{(n)}_{i_n}\otimes \BF{Y}^{(n)}_{i_n}
	\in\BB{R}^{R^X_{n-1}R^Y_{n-1} \times R^X_nR^Y_n}, 
	\quad n=1,\ldots,N. 
	\end{split}
	\end{equation}

The computational cost for calculating the contraction $\langle \ten{X},\ten{Y}\rangle$ of two TT tensors can be reduced to $\CL{O}(NIR^3)$ where $I=\max(I_n)$ and $R=\max(\{R^X_n\},\{R^Y_n\})$, which is linear with $N$ \cite{Ose2011}.

\subsection{Matrix-by-vector product}
\label{sec:mat_by_vec}

The matrix-by-vector product can also 
be efficiently represented by vector TT and matrix TT decompositions. 
Let $\BF{x} \in\BB{R}^{J_1J_2\cdots J_N}$ and $\BF{A} \in\BB{R}^{I_1I_2\cdots I_N\times J_1J_2\cdots J_N}$ be a vector and a matrix
represented by vector TT \eqref{eqn_TT_vector_skron} and 
matrix TT \eqref{eqn_TT_matrix_skron} decompositions, i.e., 
	\begin{equation*}
	\begin{split}
	\BF{x} &= \widetilde{\BF{X}}^{(1)} \skron \cdots \skron \widetilde{\BF{X}}^{(N)}, \\
	\BF{A} &= \widetilde{\BF{A}}^{(1)}  \skron \cdots  \skron \widetilde{\BF{A}}^{(N)},
	\end{split}
	\end{equation*}
with TT-ranks $\{R^X_n\}$ and $\{R^A_n\}$, respectively. 
The matrix-by-vector product can be represented by vector TT decomposition as 
	\begin{equation} \label{eqn_mat_vec_skron}
	\BF{Ax} = \BF{A} \modcon \BF{x} = 
	\left( \widetilde{\BF{A}}^{(1)} \modcon \widetilde{\BF{X}}^{(1)} \right) \skron \cdots  \skron \left( \widetilde{\BF{A}}^{(N)} \modcon \widetilde{\BF{X}}^{(N)} \right). 
	\end{equation}
That is, the $n$th block matrix of the vector TT decomposition
\eqref{eqn_mat_vec_skron} of the product $\BF{Ax}$ is expressed by 
contracted product of the corresponding block matrices. It can be re-written as
	$$
	\widetilde{\BF{A}}^{(n)} \modcon \widetilde{\BF{X}}^{(n)}
	= \left[ \BF{z}^{(n)}_{s_{n-1},s_n} \right] \in \BB{R}^{R^A_{n-1}R_{n-1}I_n \times R^A_nR_n}, 
	$$
where each block is a vector
	\begin{equation} \label{eqn_mat_vec_fiber}
	\BF{z}^{(n)}_{s_{n-1},s_n} = \BF{A}^{(n)}_{r^A_{n-1},r^A_n} \BF{x}^{(n)}_{r_{n-1},r_n}
	\in\BB{R}^{I_n}, 
	\quad 
	s_{n-1}=\overline{r^A_{n-1}r_{n-1}},\, s_n=\overline{r^A_nr_n},	
	\end{equation}
for all $s_{n-1}=1,\ldots,R^A_{n-1}R_{n-1}$, 
$s_{n}=1,\ldots,R^A_{n}R_{n}$, $n=1,\ldots,N$. 

In TT decomposition, the vector \eqref{eqn_mat_vec_fiber} 
is considered as a fiber of a 3rd-order TT-core $\ten{Z}^{(n)}\in\BB{R}^{R^A_{n-1}R_{n-1} \times I_n \times R^A_{n}R_{n}}$, i.e., $\BF{z}^{(n)}_{s_{n-1},s_n} = \ten{Z}^{(n)}(s_{n-1}, : , s_n)$, e.g., see Section~\ref{sec_TT_vec_mat}. Slice matrices of the TT-core $\ten{Z}^{(n)}$ can be written as 
	\begin{equation} \label{eqn:core_contraction_entry}
	\BF{Z}^{(n)}_{i_n} = \ten{Z}^{(n)}(:, i_n , :)
	= \sum_{j_n=1}^{J_n} \BF{A}^{(n)}_{i_n,j_n}\otimes 
	\BF{X}^{(n)}_{j_n}
	\in\BB{R}^{R^A_{n-1}R_{n-1} \times R^A_nR_n}, 
	\end{equation}
where $\BF{A}^{(n)}_{i_n,j_n} = \ten{A}^{(n)}(:,i_n,j_n,:)$ and $\BF{X}^{(n)}_{j_n} = \ten{X}^{(n)}(:,j_n:)$ are the slice matrices of the $n$th TT-cores of $\BF{A}$ and $\BF{x}$, respectively. As a result, the entries of the product $\BF{z}=\BF{Ax}$ can be  
written as  products of slice matrices of TT-cores
	\begin{equation*} 
	\BF{z}(\overline{i_1,\ldots,i_N})
	= \BF{Z}^{(1)}_{i_1}\BF{Z}^{(2)}_{i_2}\cdots\BF{Z}^{(N)}_{i_N}. 
	\end{equation*}
Note that $\BF{Z}^{(1)}_{i_1}\in\BB{R}^{1\times R^A_1R_1}$
and $\BF{Z}^{(N)}_{i_N}\in\BB{R}^{R^A_{N-1}R_{N-1} \times 1}$
are row and column vectors. 

The computational cost for computing a matrix-by-vector product 
of a matrix $\BF{A}\in\BB{R}^{I_1\cdots I_N\times J_1\cdots J_N}$ with 
a vector $\BF{x}\in\BB{R}^{J_1\cdots J_N}$ 
can be $\CL{O}(NI^2R^4)$ where $I=\max(\{I_n\},\{J_n\})$, 
$R=\max(\{R_n\},\{R^A_n\})$, by using TT decomposition \cite{Ose2011}. 

%
%

On the other hand, recall that, in \eqref{eqn:frame_01} and \eqref{eqn_vectorize_3}, 
	$$
	\BF{x} = \text{vec}\left( \ten{X} \right) = \BF{X}^{\neq n}\BF{x}^{(n)}, 
	$$
where $\BF{X}^{\neq n}$ is the frame matrix and 
$\BF{x}^{(n)}=\text{vec}(\ten{X}^{(n)})$. 
A large-scale matrix-by-vector multiplication can be reduced to a smaller matrix-by-vector multiplication as
$\BF{Ax}=\BF{AX}^{\neq n}\BF{x}^{(n)}\equiv\widetilde{\BF{A}}_{n}\BF{x}^{(n)}$, 
where 
	\begin{equation}\label{eqn:local_matrix}
	\widetilde{\BF{A}}_n = 
	\BF{AX}^{\neq n} 
	\in\BB{R}^{I_1I_2\cdots I_N\times R_{n-1}J_nR_n}. 
	\end{equation}
We often cannot calculate the matrix $\widetilde{\BF{A}}_n$ by matrix-by-matrix multiplication \eqref{eqn:local_matrix}
for a large matrix $\BF{A}$ due to high storage and computational costs. Instead, by using the distributed representation \eqref{eqn_mat_vec_skron}, a matrix-by-vector product, $\widetilde{\BF{A}}_n \BF{w}^{(n)}$ for some vector $\BF{w}^{(n)}$, can be calculated
by recursive core contractions as follows. 

\begin{prop}
Let a vector $\BF{x}\in\BB{R}^{J_1\cdots J_N}$ and a matrix  $\BF{A}\in\BB{R}^{I_1\cdots I_N \times J_1\cdots J_N}$ 
be represented by vector TT and matrix TT decompositions
with block matrices $\widetilde{\BF{X}}^{(n)}$ and $\widetilde{\BF{A}}^{(n)}$, 
respectively. 
For a fixed $n=1,\ldots,N$, let $\widetilde{\BF{A}}_n$ be the matrix defined by \eqref{eqn:local_matrix}. 
Then, for any vector $\BF{w}^{(n)} \in\BB{R}^{R_{n-1} J_n R_n}$, 
	$$
	\widetilde{\BF{A}}_n \BF{w}^{(n)} = 
	\widetilde{\BF{Z}}^{(1)}
	\skron \cdots  \skron 
	\widetilde{\BF{Z}}^{(N)}
	\in\BB{R}^{I_1I_2\cdots I_N}, 
	$$
with block matrices
	$$
	\widetilde{\BF{Z}}^{(m)} = 
	\widetilde{\BF{A}}^{(m)} \modcon \widetilde{\BF{X}}^{(m)}
	\in\BB{R}^{R^A_{m-1}R_{m-1}I_m \times R^A_mR_m}, 
	\quad m=1,\ldots,n-1,n+1,\ldots,N,
	$$	
	$$
	\widetilde{\BF{Z}}^{(n)} = 
	\widetilde{\BF{A}}^{(n)} \modcon \widetilde{\BF{W}}^{(n)}
	\in\BB{R}^{R^A_{n-1}R_{n-1}I_n\times R^A_nR_n},
	$$	
where $\ten{W}^{(n)}\in\BB{R}^{R_{n-1}\times J_n\times R_n}$ is 
the 3rd-order tensor such that $\BF{w}^{(n)} = \text{vec}(\ten{W}^{(n)})$, and  $\widetilde{\BF{W}}^{(n)} = [\ten{W}^{(n)}(r_{n-1}, : , r_n)] \in\BB{R}^{R_{n-1}J_n\times R_n}$ is the block matrix partitioned with the fiber vectors $\ten{W}^{(n)}(r_{n-1}, : , r_n) \in\BB{R}^{J_n}$. 
\end{prop}
\begin{proof}
Let $\ten{X}^{(n)} \in\BB{R}^{R_{n-1}\times J_n\times R_n}$ denote 
the TT-cores corresponding to the block matrices 
$\widetilde{\BF{X}}^{(n)} \in\BB{R}^{R_{n-1}J_n\times R_n}$. 
Let $\ten{W}$ be an $N$th-order tensor defined by
	$\ten{W} = \ten{X}^{(1)}\modcon\cdots\modcon \ten{X}^{(n-1)}
		\modcon \ten{W}^{(n)} \modcon \ten{X}^{(n+1)}\modcon
		\cdots\modcon\ten{X}^{(N)}, 
	$ 
then
	$$
	\BF{w}\equiv \text{vec}\left(\ten{W}\right) = \BF{X}^{\neq n} \BF{w}^{(n)}. 
	$$
As a result, we have 
	$$
	\widetilde{\BF{A}}_n \BF{w}^{(n)}
	= \BF{A}\BF{X}^{\neq n}\BF{w}^{(n)}
	= \BF{Aw}. 
	$$
Note that $\BF{w}=\text{vec}(\ten{W})$ can be represented by vector TT decomposition 
with block matrices $\widetilde{\BF{X}}^{(1)},\ldots,\widetilde{\BF{W}}^{(n)},\ldots,
\widetilde{\BF{X}}^{(N)}$. The result follows from the expression \eqref{eqn_mat_vec_skron}. 
\end{proof}

Figure \ref{Fig:loc_lin_op} illustrates the tensor network diagram
for the product $\widetilde{\BF{A}}_n \BF{x}^{(n)}$. For each 
$k=1,\ldots,n-1,n+1,\ldots,N,$ the node for the TT-core $\ten{A}^{(k)}$
is connected to the node for $\ten{X}^{(k)}$, which is 
represented as $\widetilde{\BF{A}}^{(k)} \modcon \widetilde{\BF{X}}^{(k)}$
in \eqref{eqn_mat_vec_skron}. 

\begin{figure}
\centering
\includegraphics[width=7.5cm]{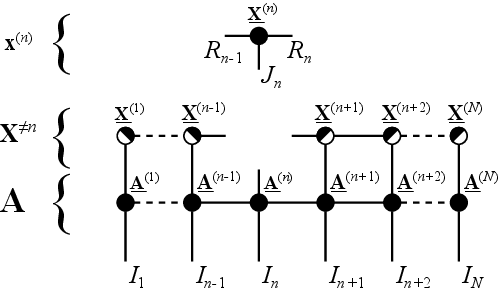}
\caption{\label{Fig:loc_lin_op}
Tensor network diagram for the matrix-by-vector product
$\widetilde{\BF{A}}_n\BF{x}^{(n)}$ represented by 
the TT decomposition \eqref{eqn_mat_vec_skron}, 
where $\BF{x}^{(n)}=\text{vec}(\ten{X}^{(n)})$. 
}
\end{figure}

\subsection{Quadratic form}

The quadratic form $\BF{x}^\Trps\BF{Ax}$ 
for a symmetric and very large-scale 
matrix $\BF{A}\in\BB{R}^{I_1I_2\cdots I_N\times I_1I_2\cdots I_N}$
can be represented by TT decomposition as follows. 
Let $\BF{x} \in\BB{R}^{I_1I_2\cdots I_N}$ and $\BF{A} \in\BB{R}^{I_1I_2\cdots I_N\times I_1I_2\cdots I_N}$ be a vector and a matrix
represented by vector TT \eqref{eqn_TT_vector_skron} and 
matrix TT \eqref{eqn_TT_matrix_skron} decompositions, i.e., 
	\begin{equation*}
	\begin{split}
	\BF{x} &= \widetilde{\BF{X}}^{(1)} \skron \cdots \skron \widetilde{\BF{X}}^{(N)}, \\
	\BF{A} &= \widetilde{\BF{A}}^{(1)}  \skron \cdots  \skron \widetilde{\BF{A}}^{(N)},
	\end{split}
	\end{equation*}
with TT-ranks $\{R^X_n\}$ and $\{R^A_n\}$, respectively. 
The quadratic form can be represented as strong Kronecker products of block matrices
	\begin{equation} \label{eqn_qform_skron}
	\BF{x}^\Trps\BF{Ax} = \BF{x} \modcon \BF{A} \modcon \BF{x} = 
	\left(\widetilde{\BF{X}}^{(1)} \modcon \widetilde{\BF{A}}^{(1)} \modcon \widetilde{\BF{X}}^{(1)} \right) \skron \cdots  \skron \left(\widetilde{\BF{X}}^{(N)} \modcon \widetilde{\BF{A}}^{(N)} \modcon \widetilde{\BF{X}}^{(N)} \right). 
	\end{equation}
That is, the $n$th block matrix in \eqref{eqn_qform_skron} is expressed by 
contracted product of the corresponding block matrices. It can be re-written as
	$$
	\widetilde{\BF{X}}^{(n)} \modcon \widetilde{\BF{A}}^{(n)} \modcon \widetilde{\BF{X}}^{(n)}
	= \left[ z^{(n)}_{s_{n-1},s_n} \right] \in \BB{R}^{R^A_{n-1}R_{n-1}^2 \times R^A_nR_n^2}, 
	$$
where each block is a scalar
	\begin{equation} \label{eqn_qform_fiber}
	z^{(n)}_{s_{n-1},s_n} = \BF{x}^{(n)\Trps}_{r_{n-1}',r_n'} \BF{A}^{(n)}_{r^A_{n-1},r^A_n} \BF{x}^{(n)}_{r_{n-1},r_n}
	\in\BB{R}, 
	\quad 
	s_{n-1}=\overline{r_{n-1}' r^A_{n-1}r_{n-1}},\, s_n=\overline{r_n' r^A_nr_n},	
	\end{equation}
for all $s_{n-1}=1,\ldots,R^A_{n-1}R_{n-1}^2$, 
$s_{n}=1,\ldots,R^A_{n}R_{n}^2$, $n=1,\ldots,N$. 

Since each block of the block matrices in \eqref{eqn_qform_skron} 
is a scalar, the quadratic form can be re-written as product of matrices as 
	$$
	\BF{x}^\Trps \BF{Ax} = \BF{Z}^{(1)} \cdots \BF{Z}^{(N)}, 
	$$
where 
	\begin{equation*} 
	\BF{Z}^{(n)} = \left[ z^{(n)}_{s_{n-1},s_n} \right] =
	\sum_{i_n=1}^{I_n} 
	\sum_{j_n=1}^{I_n} 
	\BF{X}^{(n)}_{i_n} \otimes 
	\BF{A}^{(n)}_{i_n,j_n}\otimes \BF{X}^{(n)}_{j_n},
	\quad n=1,\ldots,N, 
	\end{equation*}
where $\BF{X}^{(n)}_{i_n} = \ten{X}^{(n)}(:,i_n,:)$ and 
$\BF{A}^{(n)}_{i_n,j_n} = \ten{A}^{(n)}(:,i_n,j_n,:)$ are slice matrices of TT-cores, and 
$\BF{Z}^{(1)}\in\BB{R}^{1\times R^A_1 R_1^2}$ and 
$\BF{Z}^{(N)}\in\BB{R}^{R^A_{N-1}R_{N-1}^2\times 1}$ are row and 
column vectors, respectively.


On the other hand, recall that
	$$
	\BF{x} = \text{vec}\left( \ten{X} \right) = \BF{X}^{\neq n}\BF{x}^{(n)}, 
	$$
with $\BF{x}^{(n)} = \text{vec}(\ten{X}^{(n)})$, so the quadratic form $\BF{x}^\Trps \BF{Ax}$ reduces to 
	$$
	\BF{x}^\Trps \BF{Ax} 
	=
	\BF{x}^{(n)\Trps}
	(\BF{X}^{\neq n})^{\Trps}\BF{AX}^{\neq n}\BF{x}^{(n)}
	\equiv 
	\BF{x}^{(n)\Trps}  \overline{\BF{A}}_{n}\BF{x}^{(n)},
	$$ 
where the matrix
	\begin{equation}\label{eqn:local_form}
	\overline{\BF{A}}_{n} = 
	(\BF{X}^{\neq n})^{\Trps}\BF{AX}^{\neq n}
	\in\BB{R}^{R_{n-1}I_nR_n\times R_{n-1}I_nR_n}
	\end{equation}
is a much smaller matrix than $\BF{A}$ when TT-ranks $R_{n-1}$ and $R_n$
are sufficiently small. Since $\overline{\BF{A}}_{n}$ often cannot be 
calculated by matrix-by-matrix multiplication for a large matrix $\BF{A}$, 
we calculate it iteratively by recursive core contractions 
based on the distributed representation \eqref{eqn_qform_skron} 
as follows. 

\begin{prop}
Let a vector $\BF{x}\in\BB{R}^{I_1\cdots I_N}$ and a matrix  $\BF{A}\in\BB{R}^{I_1\cdots I_N \times I_1\cdots I_N}$ 
be represented by vector TT and matrix TT decompositions
with block matrices $\widetilde{\BF{X}}^{(n)}$ and $\widetilde{\BF{A}}^{(n)}$, 
respectively. 
For a fixed $n=1,\ldots,N$, let $\overline{\BF{A}}_n$ be the matrix defined by \eqref{eqn:local_form}. 
Then, for any vector $\BF{y}^{(n)}, \BF{w}^{(n)} \in\BB{R}^{R_{n-1} I_n R_n}$, 
	$$
	\BF{y}^{(n)\Trps} \overline{\BF{A}}_n \BF{w}^{(n)} = 
	\widetilde{\BF{Z}}^{(1)} \skron\cdots\skron \widetilde{\BF{Z}}^{(N)}
	\in\BB{R}^{R_{n-1}I_nR_n}, 
	$$
with block matrices
	$$
	\widetilde{\BF{Z}}^{(m)} = 
	\widetilde{\BF{X}}^{(m)} \modcon \widetilde{\BF{A}}^{(m)} \modcon \widetilde{\BF{X}}^{(m)}
	\in\BB{R}^{R^A_{m-1}R^2_{m-1}\times R^A_mR^2_m}, 
	\quad m=1,\ldots,n-1,n+1,\ldots,N,
	$$	
	$$
	\widetilde{\BF{Z}}^{(n)} = 
	\widetilde{\BF{Y}}^{(n)} \modcon
	\widetilde{\BF{A}}^{(n)} \modcon \widetilde{\BF{W}}^{(n)}
	\in\BB{R}^{R^A_{n-1}R^2_{n-1}I_n\times R^A_nR^2_n},
	$$	
where $\ten{Y}^{(n)}, \ten{W}^{(n)}\in\BB{R}^{R_{n-1}\times I_n\times R_n}$ are
the 3rd-order tensors such that $\BF{y}^{(n)} = \text{vec}(\ten{Y}^{(n)})$, $\BF{w}^{(n)} = \text{vec}(\ten{W}^{(n)})$, and  $\widetilde{\BF{Y}}^{(n)} = [\ten{Y}^{(n)}(r_{n-1}, : , r_n)]$, $\widetilde{\BF{W}}^{(n)} = [\ten{W}^{(n)}(r_{n-1}, : , r_n)] \in\BB{R}^{R_{n-1}I_n\times R_n}$ are the block matrices partitioned with the fiber vectors $\ten{Y}^{(n)}(r_{n-1}, : , r_n)$, $\ten{W}^{(n)}(r_{n-1}, : , r_n) \in\BB{R}^{I_n}$. 

\end{prop}
\begin{proof}
Let $\ten{X}^{(n)} \in\BB{R}^{R_{n-1}\times I_n\times R_n}$ denote 
the TT-cores corresponding to the block matrices 
$\widetilde{\BF{X}}^{(n)} \in\BB{R}^{R_{n-1}I_n\times R_n}$. 
Let $\ten{Y}$, $\ten{W}$ be $N$th-order tensors defined by
	$$\ten{Y} = \ten{X}^{(1)}\modcon\cdots\modcon \ten{X}^{(n-1)}
		\modcon \ten{Y}^{(n)} \modcon \ten{X}^{(n+1)}\modcon
		\cdots\modcon\ten{X}^{(N)}, 
	$$ 
	$$\ten{W} = \ten{X}^{(1)}\modcon\cdots\modcon \ten{X}^{(n-1)}
		\modcon \ten{W}^{(n)} \modcon \ten{X}^{(n+1)}\modcon
		\cdots\modcon\ten{X}^{(N)}, 
	$$ 
then
	$$
	\BF{y}\equiv \text{vec}\left(\ten{Y}\right) = \BF{X}^{\neq n} \BF{y}^{(n)}, 
	\quad
	\BF{w}\equiv \text{vec}\left(\ten{W}\right) = \BF{X}^{\neq n} \BF{w}^{(n)}. 
	$$
As a result, we have 
	$$
	\BF{y}^{(n)\Trps} \overline{\BF{A}}_n \BF{w}^{(n)}
	= \BF{y}^{(n)\Trps} (\BF{X}^{\neq n})^\Trps\BF{A}\BF{X}^{\neq n} \BF{w}^{(n)}
	= \BF{y}^\Trps\BF{Aw}. 
	$$
Since $\BF{y}=\text{vec}(\ten{Y})$ (resp. $\BF{w}=\text{vec}(\ten{W})$) can be represented by vector TT decomposition 
with block matrices $\widetilde{\BF{X}}^{(1)},\ldots,\widetilde{\BF{Y}}^{(n)},\ldots,
\widetilde{\BF{X}}^{(N)}$ (resp. $\widetilde{\BF{X}}^{(1)},\ldots,\widetilde{\BF{W}}^{(n)},\ldots, \widetilde{\BF{X}}^{(N)}$), the result follows from the expression \eqref{eqn_qform_skron}. 
\end{proof}

Figure \ref{Fig:loc_quad_op} illustrates the tensor network diagram 
for the quadratic form $\BF{x}^{(n)\Trps} \overline{\BF{A}}_n \BF{x}^{(n)}$. 
It is clear that each node for core tensor $\ten{A}^{(k)},k=1,\ldots,n-1,n+1,\ldots,N,$ 
is connected to the node for core tensor $\ten{X}^{(k)}$, 
which is represented as 
$\widetilde{\BF{X}}^{(k)} \modcon \widetilde{\BF{A}}^{(k)} \modcon \widetilde{\BF{X}}^{(k)}$ in 
\eqref{eqn_qform_skron}. 

\begin{figure}
\centering
\includegraphics[width=7.5cm]{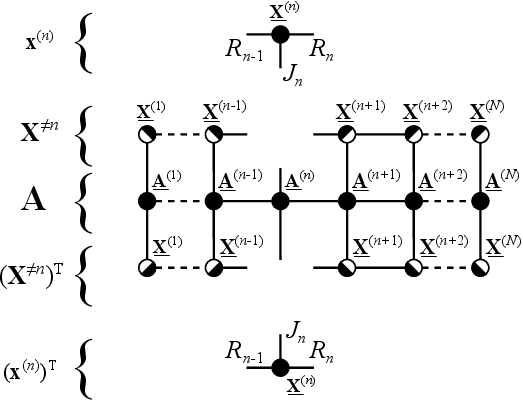}
\caption{\label{Fig:loc_quad_op}Tensor network diagram for the quadratic form
$\BF{x}^{(n)\Trps} \overline{\BF{A}}_n \BF{x}^{(n)}$
represented by the TT decomposition \eqref{eqn_qform_skron}, 
where $\BF{x}^{(n)}=\text{vec}(\ten{X}^{(n)})$. 
}
\end{figure}

%
%

\section{Discussion and conclusions}

In this paper, we proposed several extended mathematical operations on tensors
and developed their multilinear algebraic properties and their links with 
tensor network formats, especially TT decompositions. 
We generalized the standard matrix-based operations such as the 
Kronecker product, Hadamard product, and direct sum, and proposed tensor-based operations such as the partial trace and contracted product for block tensors. 
We have shown that the tensor-based operations are able to not only simplify traditional index notation for TT representations but also describe important basic operations 
which are very useful for computational algorithms using large-scale vectors, matrices, and higher-order tensors. 

The partial trace operator can be used for describing the tensor chain (TC) decomposition \cite{EspNarSch2012,Kho2012} simply by slightly modifying the suggested TT representations. Properties of TC decomposition should be more investigated 
in the future work. Moreover, the definitions and properties of partial Kronecker product, 
partial direct sum, and contracted product 
can also be generalized to any tensor network decompositions
such as hierarchical Tucker (HT) \cite{Gra2010a,Hac2012, HacKuhn2009}
and hybrid formats \cite{Gra2013,Kho2012}. 

The partial contracted products of either the left
or right core tensors of TT decomposition are matricized and 
used as a building block of the frame matrices. 
We have shown that the suggested tensor operations 
can be used to prove the orthonormality of the 
frame matrices, which have been proved 
only by using index notation in the literature. 
The developed relationships also play a key role in the 
alternating linear scheme (ALS)
and modified alternating linear scheme (MALS) 
algorithms \cite{Holtz2012} for reducing the large-scale 
optimizations to iterative smaller scale problems. 
Recent studies adjust the frame matrices in order to incorporate 
rank adaptivity and improve convergence for the ALS 
\cite{DolSav2014,KresSteinUsh2014}. 
In this work, we have derived explicit representations of 
the localized linear maps $\widetilde{\ten{A}}_n$
and $\overline{\ten{A}}_n$ by the proposed tensor operations, 
which are important for TT-based iterative methods for breaking the 
curse-of-dimensionality \cite{Cic2014a}, while the global convergence 
of the methods remains as a future work. 
In addition, it is important to keep the TT-ranks moderate
for a feasible computational cost, which is a crucial issue 
for real world applications of TT decompositions, see, e.g., \cite{Ver2014}.




\appendix

\end{document}